\documentclass[]{amsart}
\usepackage[a4paper]{geometry}

\usepackage{amsmath,amssymb,amsthm,mathrsfs,color}
\usepackage{bbm}
\usepackage{hyperref}
\hypersetup{colorlinks,%
    citecolor=black,%
    filecolor=black,%
    linkcolor=blue,%
    urlcolor=blue}
\usepackage{graphicx}
\usepackage{ifpdf}
\usepackage[small,nohug,heads=littlevee]{diagrams}
\diagramstyle[labelstyle=\scriptstyle]

\usepackage{multicol,lipsum}

\begin{document}

\newtheorem{thm}{Theorem}[section]
\newtheorem{lm}[thm]{Lemma}
\newtheorem{dfn}[thm]{Definition}
\newtheorem{prop}[thm]{Proposition}
\newtheorem{cor}[thm]{Corollary}
\newtheorem{eg}[thm]{Example}

\newcommand{\F}{\mathbf{F}}
\newcommand{\N}{\mathbf{N}}
\newcommand{\R}{\mathbf{R}}
\newcommand{\C}{\mathbf{C}}
\newcommand{\Z}{\mathbf{Z}}
\newcommand{\Q}{\mathbf{Q}}
\newcommand{\A}{\mathbf{A}}
\newcommand{\T}{\mathbf{T}}
\newcommand{\K}{\mathbf{K}}
\newcommand{\LL}{\mathcal{L}}

\newcommand{\G}{\text{G}}
\newcommand{\re}{\text{Re}}
\newcommand{\im}{\text{Im}}
\newcommand{\gal}{\text{Gal}}
\newcommand{\ke}{\text{Ker}}
\newcommand{\ma}{\text{Max}}
\newcommand{\Spec}{\text{Spec}}
\newcommand{\Pic}{\text{Pic}}
\newcommand{\ord}{\text{ord}}
\newcommand{\app}{\thickapprox}
\newcommand{\deh}{H_\mca{D}}
\newcommand{\moh}{H_\mca{M}}
\newcommand{\ab}{\text{ab}}
\newcommand{\Mp}{\text{Mp}}
\newcommand{\Sp}{\text{Sp}}
\newcommand{\GL}{\text{GL}}
\newcommand{\PGL}{\text{PGL}}
\newcommand{\SL}{\text{SL}}
\newcommand{\Ind}{\text{Ind}}
\newcommand{\Res}{\text{Res}}

\newcommand{\msc}[1]{\mathscr{#1}}
\newcommand{\mfr}[1]{\mathfrak{#1}}
\newcommand{\mca}[1]{\mathcal{#1}}
\newcommand{\mbf}[1]{\mathbbm{#1}}

\newcommand{\cu}[1]{\textsc{\underline{#1}}}
\newcommand{\set}[1]{\bigl\{#1\bigr\}}
\newcommand{\wt}[1]{\overline{#1}}
\newcommand{\angb}[2]{\big\langle #1, #2 \big\rangle}
\newcommand{\exseq}[3]{1\longrightarrow #1 \longrightarrow #2 \longrightarrow #3 \longrightarrow 1}

\title{The Residual Spectrum of $\Mp_4(\A_k)$}
\author{Fan Gao}
\address{Department of Mathematics, National University of Singapore, 10 Lower Kent Ridge Road, Singapore, 119076}
\email{gaofan@nus.edu.sg}
\date{}
\maketitle

\begin{abstract}
We compute the residual spectrum of the global metaplectic group $\Mp_4(\A_k)$, by using the theory of Eisenstein series. The residual spectra obtained are interpreted as near equivalence classes in the framework of the Arthur conjecture.
\end{abstract}

\section{Introduction}
Let $\Sp_4$ be the symplectic group of rank 2 defined over a number field $k$ with $\A_k$ its adele ring. The global metaplectic group $\Mp_4(\A_k)$ is a degree two central covering of $\Sp_4(\A_k)$. By the general theory of Eisenstein series extended to covering groups (cf. \cite{MoW}), the Hilbert space $L^2_\text{gen}(\Sp_4(k)\backslash\Mp_4(\A_k))$ of genuine functions has a spectral decomposition
\begin{align*}
L^2_\text{gen}(\Sp_4(k)\backslash \Mp_4(\A_k)) &= L^2_\text{dis} \oplus L^2_\text{cont} \\
&= L^2_\text{cusp} \oplus L^2_\text{res} \oplus L^2_\text{cont},
\end{align*}

\noindent where $L^2_\text{dis}$ is the closed invariant space spanned by irreducible subrepresentations of $L^2(\Sp_4(k)\backslash \Mp_4(\A_k))$, and it consists of the cuspidal part $L^2_\text{cusp}$ and the residue part $L^2_\text{res}$.

The residual spectrum $L^2_\text{res}(\Sp_4(k)\backslash \Mp_4(\A_k))$ can be further decomposed into
$$L^2_\text{res}(\Sp_4(k)\backslash \Mp_4(\A_k))=L^2_\text{d}(\wt{P}_1) \oplus L^2_\text{d}(\wt{P}_2) \oplus L^2_\text{d}(\wt{B}),$$
where $\wt{P_i}, i=1, 2$ are the maximal parabolic subgroups of $\Mp_4(\A_k)$ associated with the Siegel and non-Siegel maximal parabolic $P_i, i=1,2$ of $\Sp_4$ respectively, while $\wt{B}$ with the Borel subgroup $B$ of $\Sp_4$. The three components in the decomposition together contribute exhaustively to the non-cuspidal discrete spectrum, the so-called residual spectrum. They are obtained by taking the residues of Eisenstein series associated to cuspidal representations of the Levi factors of the parabolic subgroups $\wt{P}_1, \wt{P}_2, \wt{B}$ respectively.

Therefore a study of the residual spectrum inevitably involves an investigation of the poles of Eisenstein series, which in fact coincide with those of the constant terms. The latter can be expressed as sums of intertwining operators. The analysis of these intertwining operators thus leads to certain global $L$-functions associated to the aforementioned cuspidal representations, which in turn gives precise condition for the existence of poles of Eisenstein series.

The purpose of this paper is to explicate the process of analysis and computations and to determine $L^2_\text{res}$ for $\Mp_4$. For classical linear groups $\Sp_4$ and $\G_2$ such results have been obtained by H. Kim (cf. \cite{Kim1}, \cite{Kim2}) and S. Zampera (cf. \cite{Zam}). The interpretation in the framework of Arthur conjecture, for the $\G_2$ case for instance, has appeared in \cite{GGJ}. General references on the theory of Eisenstein series include \cite{MoW}, \cite{Lan}.\\

\textbf{Acknowledgment.} We are much indebted to Professor Wee Teck Gan for his guidance and numerous discussions on various topics.

\section{Preliminaries}
\subsection{Local Metaplectic Groups and Parabolic Subgroups}

In view of the global consideration later, we fix a number field $k$ henceforth as our base field. Let $W$ be a 4-dimensional symplectic vector space over $k$. We obtain the associated symplectic group $\Sp_4$ over $k$. Consider the groups $\Sp_4(k)$ and $\Sp_4(k_v)$ for any place $v$ of $k$, finite or infinite. In particular, when $k_v\ne \C$, the group $\Sp_4(k_v)$ has a unique central extension group $\Mp_4(k_v)$:
$$\exseq{\mu_2}{\Mp_4(k_v)}{\Sp_4(k_v)}.$$

One can write as sets
$$\Mp_4(k_v)=\Sp_4(k_v)\times \set{\pm 1}$$ with group operation on the right given by
$$(g_1, \epsilon_1)(g_2, \epsilon_2)=(g_1g_2,\epsilon_1\epsilon_2 c(g_1,g_2)),$$
for some 2-cocycle on $\Sp_4(k_v)$ with values in $\mu_2=\set{\pm 1}$. The more precise description of the cocycle can be found in \cite{Kud} and \cite{Rao}.

For any subset $H\subseteq \Sp_4(k_v)$, denote by $\wt{H}$ its preimage in $\Mp_4(k_v)$, which may or may not be a split covering. For example, let $B=TU$ be the Borel subgroup in $\Sp_4(k_v)$ with maximal split torus $T$:
$$T=\set{t(a,b)=\text{diag}(a,b,b^{-1},a^{-1})}.$$
Then the metaplectic covering splits over $U$ by the trivial section, and we have $\wt{U}=U\times \mu_2$ as groups. When the characteristic of the residue field of $k_v$ is odd, the covering also splits over the maximal compact subgroup $\Sp_4(\mfr{O}_v)$. Over the torus $T$, the covering map does not necessarily split, but still one knows that $\wt{T}$ is abelian.


Before we proceed, some notations are to be set up first. Let $\alpha_1, \alpha_2$ be the short and long simple root for the algebraic group $\Sp_4$, denote $\alpha_3=\alpha_1+\alpha_2, \alpha_4=2\alpha_1+ \alpha_2$ through this paper. Associated to the short and long root are the fundamental weights $\beta_1=\alpha_4/2$ and $\beta_2=\alpha_3$. The Siegel and non-Siegel maximal parabolic $P_1, P_2$ are generated by $\alpha_1$ and $\alpha_2$, with Levi-factors $\GL_2$ and $\GL_1\times \SL_2$ respectively. The Borel subgroup $B=TU$ is contained in both.

As usual, we use $w_i, i=1,2$ to represent the simple reflection in the Weyl group associated to $\alpha_i$. We choose representatives respectively for $w_1$ and $w_2$ as
$$
w_1=\left( \begin{array}{cccc}
 0 & 1 & &\\
 1 & 0 & &\\
 & & 0& 1\\
 & & 1 &0
\end{array} \right), \quad
w_2=\left( \begin{array}{cccc}
 1 & & &\\
 & 0& 1 &\\
 & -1 &0 &\\
 & &  & 1
\end{array} \right).
$$
The blank entries are understood to be filled by zeros. For any place $v$, we can view $w_i$ as an element in $\Sp_4(\mfr{O}_v)$. If $k_v$ is of odd residual characteristic, the aforementioned splitting of $\Sp_4(\mfr{O}_v)$ gives $(w_i, 1)$ as the splitting of the representative $w_i$ (cf. \cite{Szp} sect. 2.3). In this case we can write $w_i$ for $(w_i, 1)$ without introducing any ambiguity.

For convenience and later reference, we list in the table the reduced decomposition of Weyl group elements and their actions on positive roots and the torus. Throughout the paper, we will use the abbreviation $w_{ijij}$ for $w_iw_jw_jw_j$ for example, i.e. $w_{12}$ for $w_1w_2$ and $w_{121}$ for $w_1w_2w_1$.

\begin{table}
\begin{multicols}{2}
\qquad

\caption{Decomposition of Weyl elements and the actions on torus}

    \begin{tabular}{ | p{0.7in} | r | r | r | r | p{1in} |}
    \hline
    $w$ & $\alpha_1$ & $\alpha_2$ & $\alpha_3$ & $\alpha_4$ & $w^{-1}t(a,b)w$ \\ \hline
    $1$ & $\alpha_1$ & $\alpha_2$ & $\alpha_3$ & $\alpha_4$ & $t(a,b)$ \\ \hline
    $w_1$ & $-\alpha_1$ & $\alpha_4$ & $\alpha_3$ & $\alpha_2$ & $t(b,a)$ \\ \hline
    $w_2$ & $\alpha_3$ & $-\alpha_2$ & $\alpha_1$ & $\alpha_4$ & $t(a,b^{-1})$ \\ \hline
    $w_{12}$ & $\alpha_3$ & $-\alpha_4$ & $-\alpha_1$ & $\alpha_2$ & $t(b^{-1},a)$ \\ \hline
    $w_{21}$ & $-\alpha_3$ & $\alpha_4$ & $\alpha_1$ & $-\alpha_2$ & $t(b,a^{-1})$ \\ \hline
    $w_{121}$ & $-\alpha_3$ & $\alpha_2$ & $-\alpha_1$ & $-\alpha_4$ & $t(a^{-1},b)$ \\\hline
    $w_{212}$ & $\alpha_1$ & $-\alpha_4$ & $-\alpha_3$ & $-\alpha_2$ & $t(b^{-1},a^{-1})$ \\ \hline
    $w_{1212}$ & $-\alpha_1$ & $-\alpha_2$ & $-\alpha_3$ & $-\alpha_4$ & $t(a^{-1},b^{-1})$ \\
    \hline
    \end{tabular}

\goodbreak

\setlength{\unitlength}{1cm}
\begin{picture}(4.7,4.7)(-0.5,0.5)
\put(2,1){\vector(1,0){2}}
\put(2,1){\vector(1,1){2}}
\put(2,1){\vector(0,1){2}}
\put(2,1){\vector(-1,1){2}}
\put(4.1,1){$\alpha_1$}
\put(0,3.1){$\alpha_2$}
\put(1.9,3.1){$\alpha_3=\beta_2$}
\put(4.1,3){$\alpha_4$}

\put(2,1){\vector(1,1){1}}
\put(3.1,1.8){$\beta_1$}
\end{picture}
\end{multicols}
\end{table}

\hspace{0.1cm}

Naturally, the local parabolic subgroups $\wt{P}_1(k_v), \wt{P}_2(k_v), \wt{B}(k_v)$ are defined to be the preimage of $P_1(k_v), P_2(k_v)$ and $B(k_v)$ under the projection $\Mp_4(k_v) \to \Sp_4(k_v)$. If $P(k_v)=M(k_v)N(k_v)$ is a parabolic subgroup of $\Sp(k_v)$, then $\wt{P}(k_v)=\wt{M}(k_v)N(k_v)$, since the unipotent radical splits trivially. We call $\wt{M}(k_v)$ the Levi of the parabolic subgroup $\wt{P}(k_v)$ in $\Mp_4(k_v)$. For example,
$$\wt{M}_1(k_v)=\wt{\GL}_2(k_v) \text{ and } \wt{M}_2(k_v)=(\wt{\GL}_1(k_v) \times \Mp_2(k_v))/\Delta \mu_2,$$
where $\Delta$ stands for the diagonal imbedding of $\mu_2$.

\subsection{Local Representations and Parabolic Induction}

In order to define parabolic induction, one needs to understand the genuine representations of $\wt{T}(k_v), \wt{M}_1(k_v)$ and $\wt{M}_2(k_v)$. In fact, some observation shows that it suffices to understand those of $\wt{\GL}_n(k_v), n=1,2$ and $\Mp_2(k_v)$.

With a view of the global situation, we fix an additive character $\psi=\prod \psi_v$ of $\A_k$ throughout. For an arbitrary place $v$ with the additive character $\psi_v: k_v \to \C^\times$, there is a natural genuine character $\wt{\chi}_{\psi_v}$ of $\wt{\GL}_n(k_v)$ derived from the Weil's factor (cf. \cite{Wei}), such that the map
$$\tau \mapsto \wt{\tau}_v=\tau\otimes \wt{\chi}_{\psi_v}$$
gives a bijection between Irr($\GL_n(k_v)$) and the set Irr($\wt{\GL}_n(k_v)$) of genuine irreducible representations of $\wt{\GL}_n(k_v)$. Note that such bijection depends on the choice of the additive character.

For Irr($\Mp_2(k_v)$), we will refrain from giving the details but refer to \cite{GaS} for descriptions.


Based on this, we can now introduce parabolic inductions from $\wt{P}_1(k_v), \wt{P}_2(k_v)$ or $\wt{B}(k_v)$. For example, we consider the case of $\wt{P}_2(k_v)$. Let $\chi_v \in$ Irr($\GL_1(k_v)$) and $\sigma_v \in $ Irr($\Mp_2(k_v)$), then $\wt{\chi}_v \boxtimes \sigma$ is a representation of $\wt{\GL}_1(k_v)\times \wt{\Mp}_2(k_v)$ which descends to be in Irr($\wt{M}_2(k_v)$). Therefore, one may define the normalized parabolic induction

$$I_{\wt{P}_2}(\wt{\chi}_v\boxtimes \sigma_v) =\Ind_{\wt{P}_2}^{\Mp_4} \wt{\chi}_v \boxtimes \sigma_v.$$

\subsection{Rank One Intertwining Operator and Coefficients}
We collect some facts and computational results on local intertwining operators and the coefficients which appear in the unramified case. Let $\psi_v$ be the additive character given above, and let $\chi_v, \mu_v$ be two characters of $\GL_1(k_v)$. We assume that these characters are all unramified, and consider the induced representation
$$I_{\wt{B}}(\wt{\chi}_v\boxtimes\wt{\mu}_v)=\Ind_{\wt{B}}^{\Mp_4}\wt{\chi}_v\boxtimes \wt{\mu}_v.$$

For the two Weyl group elements $w_i$ associated to $\alpha_i, i=1,2$, we will compute the intertwining map
$$A(\wt{\chi}_v\boxtimes \wt{\mu}_v, w_i): I_{\wt{B}}(\wt{\chi}_v\boxtimes\wt{\mu}_v)\longrightarrow I_{\wt{B}}(w_i(\wt{\chi}_v\boxtimes\wt{\mu}_v)).$$

Recall that by abuse of notation, we use $w_i\in \Sp_4(\mfr{O}_v)$ to represent the element $(w_i,1)$ in $\Mp_4(k_v)$, though in general the splitting over $\Sp_4(\mfr{O}_v)$ is not trivial. The intertwining operator here labels essentially a rank one computation, and it is necessary to differentiate the case of short and long root, i.e. $w_1$ or $w_2$. For the short root $\alpha_1$, we have
\begin{diagram}
1 &\rTo &\mu_2  &\rTo   &\Mp_4(k_v)   &\rTo     &\Sp_4(k_v)   &\rTo    &1  \\
  &     &\uTo^=   &       &\uInto &\luInto^{g\mapsto (g, 1)}  &\uInto_{\phi_{\alpha_1}} \\
1 &\rTo &\mu_2  &\rTo   &\SL_2(k_v)\times \mu_2   &\rTo     &\SL_2(k_v)   &\rTo    &1,
\end{diagram}
where $\phi_{\alpha_1}$ is the embedding given by $\alpha_1$. That is to say, the top covering map splits over $\phi_{\alpha_1}(\SL_2(k_v))$ by the trivial section (cf. \cite{Rao}).

However, for the long root $\alpha_2$, the $\SL_2(k_v)$ embedded by $\alpha_2$ inherits a nontrivial degree two covering $\Mp_2(k_v)$. That is
\begin{diagram}
1 &\rTo &\mu_2  &\rTo   &\Mp_4(k_v)   &\rTo     &\Sp_4(k_v)   &\rTo    &1  \\
  &     &\uTo^=   &       &\uInto       &  &\uInto_{\phi_{\alpha_2}} \\
1 &\rTo &\mu_2  &\rTo   &\Mp_2(k_v)   &\rTo     &\SL_2(k_v)   &\rTo    &1.
\end{diagram}

To compute the local coefficients of intertwining operators for unramified representations, it reduces to $\SL_2$ or $\Mp_2$ case. Though the first case is familiar, we recall both to illustrate the difference. In this paper, we will use $B_o$ to denote the Borel subgroup for either $\SL_2$ or $\GL_2$, in which case the nontrivial Weyl element is written as $w$. No confusion will arise from the context.

First, let $f_v\in I_{B_o}^{\SL_2}(|\ |^s\mu_v)$ be a normalized $\Sp_2(\mfr{O}_v)$-fixed vector, and let $f_v^\prime \in I_{B_o}^{\SL_2}(w(|\ |^s\mu_v))$ be the corresponding vector. Then in the $\SL_2$ case the intertwining operator $A(s, \mu_v)$ gives
$$A(s, \mu_v)f_v=\frac{L(s,\mu_v)}{L(s+1,\mu_v)} f_v^\prime.$$

However, suppose $f_v\in I_{\wt{B}_o}^{\Mp_2}(|\ |^s\wt{\mu}_v)$ is a normalized $\Sp_2(\mfr{O}_v)$-fixed vector, and let $f_v^\prime \in I_{\wt{B}_o}^{\Mp_2}(w(|\ |^s\wt{\mu}_v))$ be the corresponding vector. Then the intertwining operator $A(s, \wt{\mu}_v)$, for the $\Mp_2$ case, gives
$$A(s, \wt{\mu}_v)f_v=\frac{L(2s,\mu_v^2)}{L(2s+1,\mu_v^2)} f_v^\prime.$$

We refer to \cite{Szp} sect. 8 for the details of the computation.

\subsection{Global Metaplectic Group}
Keep notations as above, we give some basics on our underlying object of interest: the global metaplectic group $\Mp_4(\A_k)$. Fix the number field $k$, let $\Mp_4(k_v)$ be the local double cover of $\Sp_4(k_v)$ given before, consider the restricted direct product $\prod_v^\prime\Mp_4(k_v)$ with respect to the family of split open compact subgroup $\Sp_4(\mfr{O}_v)$. This is a central extension of $\Sp_4(\A_k)$ by the compact group $\prod_v \mu_{2,v}$. Let
$$Z_o=\set{(\epsilon_v)\in \prod_v\mu_{2,v}: \prod_v \epsilon_v=1}$$
and set $$\Mp_4(\A_k)=\Bigl(\prod_v^\prime\Mp_4(k_v) \Bigr)/Z_o.$$
Thus we obtain the global metaplectic group which splits over $\Sp_4(k)$ from the exact sequence
$$\exseq{\mu_2}{\Mp_4(\A_k)}{\Sp_4(\A_k)}.$$

Following this, we may define the space $\msc{A}(\Mp_4(\A_k))$ of genuine automorphic forms on $\Sp_4(k)\backslash \Mp_4(\A_k)$. Note that an irreducible genuine representation of $\Mp_4(\A_k)$ is of the form
$$\eta=\otimes \eta_v,$$
where $\eta_v$ is an irreducible genuine representation of $\Mp_4(k_v)$. In particular, every irreducible automorphic representation is of this form. As alluded in the introduction, this paper will focus on the square integrable genuine forms
$$L^2_\text{gen}(\Sp_4(k)\backslash \Mp_4(\A_k)),$$
and in fact only its residual spectrum $L^2_\text{res}$ obtained by taking residues of Eisenstein series. We refer to \cite{MoW} and \cite{GGP} for discussions of metaplectic forms on $\Mp_4(\A_k)$ and some conjectures on $L_\text{disc}^2$ in the spirit of Arthur conjectures for classical groups (\cite{Art}) respectively.


\section{Decomposition for the Siegel parabolic subgroup}
\subsection{Eisenstein Series on Metaplectic Groups}
For this section, we have the decomposition $P_1=M_1N_1$ for the Siegel parabolic $P_1$ inside $\Sp_4$, with the Levi factor $M_1\cong \GL_2$. Let $\mfr{a}^*_{P_1}=X(M_1)\otimes \R =\R\beta_2$ and $\rho_{P_1}=\frac{3}{2}\beta_2$ be the half sum of roots generating $N_1$. In fact, we can view $\beta_2$ as the determinant map on $M_1$, and will identify $s\in \C$ with $s\beta_2$.

Let $\wt{\tau}=\otimes_v \wt{\tau}_v$ be a cuspidal representation of $\wt{M}_1(\A_k)=\wt{\GL}_2(\A_k)$, where $\tau=\otimes_v \tau_v$ is a cuspidal representation of $\GL_2(\A_k)$. The local component $\wt{\tau}_v$ is described as before. Let $\phi \in \Ind_{\wt{P}_1}^{\Mp_4}\wt{\tau}$ be an element in the representation space, identified with certain $\C$-valued function on $N_1(\A_k)M_1(k)\backslash \Mp_4(\A_k)$ (cf. \cite{MoW} II.1).

Define the Eisenstein series on $\Mp_4(\A_k)$ by

$$E(s,\wt{g},\phi,\wt{P}_1)=\sum_{\gamma\in P_1(k)\backslash\Sp_4(k)} \phi(\gamma \wt{g})\text{exp}\angb{s+\rho_{P_1}}{H_{\wt{P}_1}(\gamma\wt{g})},$$
where $H_{\wt{P}_1}(\wt{g})= H_{P_1}(g)$ is the Harish-Chandra homomorphism. For convenience, we may let $\Phi_s(\wt{g})=\phi(\wt{g})\text{exp}\angb{s+\rho_{P_1}}{H_{\wt{P}_1}(\wt{g})}$ and write $E(s,\wt{g},\Phi_s,\wt{P}_1)$ for $E(s,\wt{g},\phi,\wt{P}_1)$ sometimes. Note that for fixed $s$ the representation space generated by $\Phi_s$ is equivalent to $I_{\wt{P}_1}(s,\wt{\tau})$.

We refer to \cite{MoW} for general properties of the Eisenstein series. In particular, $E(s,\wt{g},\phi,\wt{P}_1)$ converges for Re$s\gg 0$ and extends to a meromorphic function of $s\in \C$. The poles of the Eisenstein series coincide with those of its constant term along $\wt{P}_1$, which is given by

$$E_{\wt{P}_1}(s,\wt{g},\phi,\wt{P}_1)=\sum_{w\in \Omega_1} M(s, \wt{\tau}, w) \Phi_s(\wt{g}),$$
where $\Omega_1=\set{1,w_{212}}$. The element $w_i\in \Sp_4(k)$ is considered to be in $\Mp_4(\A_k)$, since the group $\Sp_4(k)$ splits into the latter. We also write $n\in N_1(\A_k)$ for the splitting $(n,1)\in \Mp_4(\A_k)$.

More generally, for any Weyl group element $w$ and any $f\in I_{\wt{P}_1}(s,\wt{\tau})$, the representation space of $\Mp_4(\A_k)$ generated by $\Phi_s$, we consider the intertwining operator $M(s, \wt{\tau}, w)f(\wt{g})$. It is given by the integral
$$M(s, \wt{\tau}, w)f(\wt{g})=\int_{N^w_1(\A_k)}f(w^{-1}n\wt{g})dn,$$
which converges for Re$s\gg0$.

In general we define $N^w_1=U\cap w N_1^\prime w^{-1}$ and $N_1^\prime$ is the unipotent group opposite to $N_1$. Here in this section we will fix $w=w_{212}$, the nontrivial element in $\Omega_1$. In this case, we have $N_1^w=N_1$. It now follows from the definition of $I_{\wt{P}_1}(s,\wt{\tau})$ as a restricted tensor product that
$$M(s,\wt{\tau}, w)=\otimes A(s, \wt{\tau}_v, w), \text{ with } A(s, \wt{\tau}_v, w)f_v(\wt{g}_v)= \int_{N_1^w(k_v)}f_v(w^{-1}n_v\wt{g}_v)dn,$$
where $f=\otimes f_v$ and $f_v$ is the unique $k_v$-fixed function with normalization $f_v(e_v)=1$ for almost all $v$.

\subsection{Computation of Local Intertwining Operator}
By virtue of the tensor decomposition for $M(s,\wt{\tau}, w)$, we need to analyze both the local intertwining operator $A(s, \wt{\tau}_v, w)$ for the almost all unramified places $v$ and the other finitely many exceptions $v\in S$.

To facilitate subsequent computations for the non-Siegel and Borel case, it is helpful to obtain the explicit formulas, at least in the unramified case. Therefore we assume $\wt{\tau}_v$ is unramified, i.e. $\wt{\tau}_v \hookrightarrow I_{\wt{B}_o}^{\wt{\GL}_2}\wt{\chi}_v\boxtimes \wt{\mu}_v$ for unramified character $\chi_v, \mu_v$ of $\GL_1(k_v)$.
By inducing in stages, $I_{\wt{P}_1}(s,\wt{\tau}_v)\subseteq I_{\wt{B}}(s\beta_2,\wt{\chi}_v\boxtimes \wt{\mu}_v)$. Then $A(s, \wt{\tau}_v, w)= A(s\beta_2, \wt{\chi}_v\boxtimes \wt{\mu}_v, w)$ on their common domain. Consider the unramified vector $f_v\in I_{\wt{B}}(s\beta_2,\wt{\chi}_v\boxtimes \wt{\mu}_v)$ with
$$f_v(\wt{t}nq)=(\wt{\chi}_v\boxtimes \wt{\mu}_v)(\wt{t})\text{exp}\angb{s\beta_2+\rho_B}{H_{\wt{B}}(\wt{t})},$$
for $t\in \wt{T}(k_v), n\in N(k_v), q\in GL_2(\mfr{O}_v)$.

We need to calculate
$$A(s,\wt{\chi}_v\boxtimes \wt{\mu}_v, w) f_v(e)=\int_{N_1^w(k_v)}f_v(w^{-1}n)dn.$$

 The computation is reduced to the $\SL_2$ and $\Mp_2$ case by using the cycle relation of intertwining operators. Let $\Lambda\in \mfr{a}^*_\C:=\C\beta_1\oplus \C\beta_2, \chi_v\boxtimes \mu_v$ a character of $\wt{T}(k)$, let $A(\Lambda, \wt{\chi}_v\boxtimes \wt{\mu}_v,w)$ be the intertwining operator from $I_{\wt{B}}(\Lambda,\wt{\chi}_v\boxtimes \wt{\mu}_v)$ to $I(w\Lambda,w(\wt{\chi}_v\boxtimes \wt{\mu}_v))$. 

By using the result for $\SL_2$ and $\Mp_2$ intertwining operator of last section, we have the following
\begin{prop}\label{prop: main}
(Gindikin-Karpelevich formula) Let $\wt{\chi}_v\boxtimes \wt{\mu}_v$ be an unramified character of $\wt{T}\subseteq \Mp_4(k_v)$, and let $\Lambda\in \mfr{a}^*_\C$. Then
$$A(\Lambda, \wt{\chi}_v\boxtimes \wt{\mu}_v, w)f_v^o(e)=\prod_{\substack{\beta> 0, w\beta< 0 \\ \beta \text{ short}}} \frac{L(\angb{\Lambda}{\beta^\vee}, (\chi_v\boxtimes\mu_v) \beta^\vee)}{L(1+\angb{\Lambda}{\beta^\vee}, (\chi_v\boxtimes\mu_v) \beta^\vee)} \prod_{\substack{\beta> 0, w\beta< 0 \\ \beta \text{ long}}} \frac{L(2\angb{\Lambda}{\beta^\vee}, ((\chi_v\boxtimes\mu_v) \beta^\vee)^2)}{L(1+2\angb{\Lambda}{\beta^\vee}, ((\chi_v\boxtimes\mu_v) \beta^\vee)^2)}.$$
\end{prop}

 Back to the Siegel parabolic case,  consider $w=w_{212}$ the nontrivial element in $\Omega_1$. We have $\Lambda=s\beta_2=s\alpha_3$. Substituting $\Lambda$ into the above formula and combining all local intertwining operators, the nontrivial term in the constant term of Eisenstein series along $\wt{P}_1$ is given by
$$M(s,\wt{\tau}, w_{212})f=\frac{L^S(2s,\tau,\text{Sym}^2)}{L^S(2s+1,\tau,\text{Sym}^2)}  \bigotimes_{v\notin S} f'_v \otimes \bigotimes_{v\in S} A(s,\wt{\tau}_v, w_{212}) f_v.$$

In order to determine the pole of $M(s, \wt{\tau}, w_{212})$, we need to analyze the partial $L$-functions $L^S$ and the local intertwining maps for the exceptional places. Instead, we first consider the complete $L(2s,\tau,\text{Sym}^2)$ with its local factors given by Shahidi, it is well-known that the symmetric square $L$-function has a pole of order 1 if and only if
$$\tau\cong \tau^\vee \text{ with a nontrivial quadratic central character }\omega_\tau,$$
where the $\tau^\vee$ stands for the contragradient of $\tau$.

Second, we carry out the analysis of $A(s, \wt{\tau}_v, w_{212}), v\in S$. We make the normalization 
$$r(s, \wt{\tau}_v, w_{212})=\frac{L(2s,\tau_v,\text{Sym}^2)}{L(2s+1,\tau_v,\text{Sym}^2)\varepsilon(2s,\tau_v,\text{Sym}^2)}, \quad A(s, \wt{\tau}_v, w_{212})=r(s, \wt{\tau}_v, w_{212}) R(s, \wt{\tau}_v, w_{212}).$$

\begin{lm}\label{lm: siegel, local operator holomorphic}
For each $v\in S$, $R(s, \wt{\tau}_v, w_{212})$ can be continued to a holomorphic function for $Re(s)\ge 0$.
\end{lm}
\begin{proof}  For Re$(s)\ge 0$, it suffices to show that the poles of $r(s, \wt{\tau}_v, w_{212})$ and $A(s, \wt{\tau}_v, w_{212})$ cancel each other.

If $\tau_v$ is tempered, \cite{Sha} gives that the factor $L(2s,\tau_v,\text{Sym}^2)$ and therefore $r(s, \wt{\tau}_v, w_{212})$ has no pole or zero for Re$s\ge 0$. Also in this case, the holomorphicity of $A(s, \wt{\tau}_v, w_{212})$ follows from \cite{BoW} section IV. 4 and \cite{Sil} section 5.4, with slight modification of the argument. Therefore we assume that $\tau_v$ is a complementary series, and without loss of generality of the form $\tau_v=\Ind_B^{\GL_2}\nu|\ |^a\times \nu|\ |^{-a}$, where $\nu$ is a unitary character and $a\in \R^+$. In fact, we know $a\in (0,1/9)$ by \cite{KiS}.

By the Langlands-Shahidi method and the multiplicativity of $\gamma$-factor which is used to define $L(2s,\tau_v,\text{Sym}^2)$ (cf. \cite{Sha}), we see that
$$L(2s,\tau_v,\text{Sym}^2)=L(2s-2a,\nu^2)L(2s,\nu^2)L(2s+2a,\nu^2).$$
Note that the notation we use for the local factor could be unifying, both for archimedean and non-archimedean places. However, for elaboration purpose, we give details separately. Clearly for Re$(s)\ge 0$, $L(2s+1,\tau_v,\text{Sym}^2)$ has no pole or zero, therefore it is reduced to show that the poles of $L(2s,\tau_v,\text{Sym}^2)$ and $A(s, \wt{\tau}_v, w_{212})$ cancel each other.

For non-archimedean place $v$, it follows from above formula that for Re$(s)\ge 0$, the pole of $L(2s,\tau_v,\text{sym}^2)$ occurs only at $s=a$ for $\nu^2=\mbf{1}$, or $s=0$ for $\nu^2=\mbf{1}$. On the other hand, consider $A(s, \wt{\tau}_v, w_{212}): I_{\wt{B}}(\wt{\nu}|\ |^{a+s}\boxtimes \wt{\nu}|\ |^{-a+s})\longrightarrow I_{\wt{B}}(\wt{\nu}^{-1}|\ |^{a-s}\boxtimes \wt{\nu}^{-1}|\ |^{-a-s})$. It factorizes as
\begin{diagram}
I_{\wt{B}}(\wt{\nu}|\ |^{a+s}\boxtimes \wt{\nu}|\ |^{-a+s}) &\rTo^{A(w_2)} &I_{\wt{B}}(\wt{\nu}|\ |^{a+s}\boxtimes \wt{\nu}^{-1}|\ |^{a-s}) &\rTo^{A(w_1)} &I_{\wt{B}}(\wt{\nu}^{-1}|\ |^{a-s}\boxtimes \wt{\nu}|\ |^{a+s}) &\rTo^{A(w_2)} &I_{\wt{B}}(\wt{\nu}^{-1}|\ |^{a-s}\boxtimes \wt{\nu}^{-1}|\ |^{-a-s}).
\end{diagram}
The first and third operators are both $\Mp_2$ rank-one operator, while the middle is a $\GL_2$ intertwining operator. The first operator $A(w_2)$ has a simple pole at $s=a$ when $\nu^2=\mbf{1}$, the second operator at $s=0$ when $\nu^2=\mbf{1}$, and the third operator has a simple pole at $s=-a$ when $\nu^2=\mbf{1}$. Therefore for Re$s\ge 0$, the only pole for $A(s, \wt{\tau}_v, w)$ occurs at $s=a$ when $\nu^2=\mbf{1}$, or $s=0$ when $\nu^2=\mbf{1}$.

For archimedean place $v$ real say, the character $\nu$ is either trivial or the sign character. The local $L$ factor in terms of Gamma function reads
$$L_\infty(2s,\tau_v, \text{Sym}^2)=\pi^{-3s}\Gamma(s-a)\Gamma(s)\Gamma(s+a).$$
In either case we see the only possible pole of the $L$-factor occurs at $s=a$ for $\nu^2=\mbf{1}$ or $s=0$ for $\nu^2=\mbf{1}$, which agrees with that of the operator $A(s, \wt{\tau}_v, w_{212})$. The complex case follows from similar formula as the real case. The proof is completed.

\end{proof}

To determine the image of the normalized local intertwining operator we have
\begin{lm}\label{lm: image irred}
For each $v$, the image of $R(1/2, \wt{\tau}_v, w_{212})$ is irreducible and nonzero for all place $v$ of $k$.
\end{lm}
\begin{proof} In all cases considered below, the normalizing factor $r(s, \wt{\tau}_v, w_{212})$ has no pole or zero at $s=1/2$, therefore, it suffices to prove the lemma for the operator $A(1/2, \wt{\tau}_v, w_{212})$.

When $\tau_v$ is tempered, it follows from the Langlands classification theorem (cf. \cite{BaJ}). Now suppose that $\tau_v=\Ind_{B_o}\nu|\ |^a\boxtimes \nu|\ |^{-a}$ is a complementary series representation, we can assume $a\in (0,1/9)$ by \cite{KiS}. By inducing by steps, 
$$\Ind_{\wt{P}_1}^{\Mp_4}(s\beta_2, \wt{\tau}_v)=\Ind_{\wt{T}}^{\Mp_4}(s\beta_2+a\alpha_1, \wt{\nu}\boxtimes \wt{\nu})=I(\wt{\nu}|\ |^{s+a}\boxtimes \wt{\nu}|\ |^{s-a})$$
 and $A(s\beta_2,\wt{\tau}_v,w_{212})=A(s\beta_2+a\alpha_1, \wt{\nu}\boxtimes \wt{\nu}, w_{212})$. Consider the chain of intertwining maps:
\begin{diagram}
I_{\wt{B}}(\wt{\nu}|\ |^{s+a}\boxtimes \wt{\nu}|\ |^{s-a}) &\rTo^{A(s\beta_2+a\alpha_1, w_{212})} &I_{\wt{B}}(\wt{\nu}^{-1}|\ |^{a-s}\boxtimes \wt{\nu}^{-1}|\ |^{-s-a}) &\rTo^{A(w_{212}(s\beta_2+a\alpha_1), w_1)} &I_{\wt{B}}(\wt{\nu}^{-1}|\ |^{-s-a}\boxtimes \wt{\nu}^{-1}|\ |^{a-s}).
\end{diagram}

The last map $A(w_{212}(s\beta_2+a\alpha_1), w_1)$ is the functorial lift of the rank-one operator
\begin{diagram}
I_{\wt{B}_o}^{\wt{\GL}_2}(\wt{\nu}^{-1}|\ |^{a-s}\boxtimes \wt{\nu}^{-1}|\ |^{-s-a}) &\rTo &I_{\wt{B}_o}^{\wt{\GL}_2}(\wt{\nu}^{-1}|\ |^{-s-a}\boxtimes \wt{\nu}^{-1}|\ |^{a-s}).
\end{diagram}
which is an isomorphism. Therefore $A(w_{212}(s\beta_2+a\alpha_1), w_1)$ is an isomorphism. On the other hand, the chain composition above yields $A(s\beta_2+a\alpha_1, \wt{\nu}\boxtimes \wt{\nu}, w_{212}w_1)$ whose image by Langlands classification is irreducible when $s=1/2$. Therefore, the image of $A(1/2, \wt{\tau}_v, w_{212})$ is isomorphic to the Langlands quotient and thus irreducible. In particular, it is nonzero.
\end{proof}

This shows that the image of $R(1/2, \wt{\tau}_v, w_{212})$ is the Langlands quotient $J_{\wt{P}_1}(1/2,\wt{\tau}_v)$ with initial data described as in the lemma.
Therefore, we have 
$$\Res_{s=1/2}E_{\wt{P}_1}(s,\wt{g},\Phi_s,\wt{P}_1)=\bigotimes_v J_{\wt{P}_1}(1/2,\wt{\tau}_v).$$
Taking residue of Eisenstein series commutes with taking the constant term, as in
\begin{diagram}
I_{\wt{P}_1}(s,\wt{\tau}) & \rTo^{\Phi_s\mapsto \Res_{s=s_o}E(\Phi_s)} & \msc{A}^2(\Sp_4(k)\backslash \Mp_4(\A_k)) \\
 & \rdTo_{\Phi_s \mapsto \Res_{s=s_o}E_{\wt{P}_1}(\Phi_s)} & \dTo_{\text{take const. term}} \\
& &\msc{A}(N_1(\A_k)T(k)\backslash \Mp_4(\A_k)).
\end{diagram}

We know that the right vertical map is injective when restricted to the image of the top map $\Res_{s=s_o}E(-)$. The Eisenstein series $E(s,\wt{g}, \Phi_s,\wt{P}_1)$ has a pole at $s=1/2$, and we consider the space spanned by its residues at $s=1/2$. Since $w(\beta_2/2)=-\alpha_3/2$, it follows from the Langlands' criterion (cf. \cite{MoW} I.4 ) that these residues are square integrable. By above injectivity, the residue $\Res_{s=1/2}E(s,\wt{g},\Phi_s,\wt{P}_1)$ can be identified with $\bigotimes_v J_{\wt{P}_1}(1/2,\wt{\tau}_v)$, which is irreducible and nonzero. 

We will make implicit identification of this kind for the non-Siegel and Borel case later without mentioning the details again.

Hence, for the Siegel parabolic case we have shown
\begin{thm}\label{thm: siegel}
Let $\msc{A}_\text{cusp}(\GL_2(\A_k))$ denote the cuspidal representations on $\GL_2(\A_k)$ and $\msc{S}\subseteq \msc{A}_\text{cusp}(\GL_2(\A_k))$ be defined as
$$\msc{S}=\set{\tau: \tau\cong \tau^\vee, \omega_\tau\ne \mbf{1}}.$$
From above, we get $J_{\wt{P}_1}(1/2,\wt{\tau}_v)$ the irreducible representation as the image of $R(1/2,\wt{\tau}_v,w)$. Then the representation $J_{\wt{P}_1}(1/2,\wt{\tau})=\bigotimes_v J_{\wt{P}_1}(1/2,\wt{\tau}_v)$ occurs in the residual spectrum of $\Mp_4(\A_k)$. In fact
$$L^2_\text{d}(\wt{P}_1)=\bigoplus_{\tau\in \msc{S}} J_{\wt{P}_1}(1/2,\wt{\tau}),$$
where $\tau$ runs over all representations in $\msc{S}$.
\end{thm}

\section{Decomposition for the non-Siegel maximal parabolic subgroup}
In this section, we consider the non-Siegel maximal parabolic with decomposition $P_2=M_2N_2$ inside $\Sp_4$, where $M_2 \cong \GL_1\times \SL_2$. Recall $\beta_1=\alpha_4/2$. Let $\mfr{a}^*_{P_2}=X(M_2)\otimes \R =\R\beta_1$ and $\rho_{P_2}=2\beta_1$ be the half sum of roots generating $N$. We identify $s\in \C$ with $s\beta_1$.

Let $\chi$ be a unitary Hecke character of $\GL_1(\A_k)$ and $\sigma$ be a cuspidal representation of $\Mp_2(\A_k)$. Then it gives a genuine cuspidal representation $\wt{\chi}\boxtimes \sigma$ of $\wt{M}_2(\A_k)$, and conversely every genuine representation of $\wt{M}_2(\A_k)$ arises in this way.

Given a function $\phi \in \Ind_{\wt{P}_2}^{\Mp_4}\wt{\chi}\boxtimes \sigma$, we define the associated Eisenstein series on $\Mp_4(\A_k)$ by

$$E(s,\wt{g}, \phi,\wt{P}_2)=\sum_{\gamma\in P_2(k)\backslash \Sp_4(k)} \phi(\gamma \wt{g})\text{exp}\angb{s+\rho_{P_2}}{H_{\wt{P}_2}(\gamma \wt{g})},$$
where $H_{\wt{P}_2}(\wt{g})= H_{P_2}(g)$ as before is the Harish-Chandra homomorphism on $\wt{P}_2$. We write $\Phi_s=\phi(\wt{g})\text{exp}\angb{s+\rho_{P_2}}{H_{\wt{P}_2}(\wt{g})}$. The poles of the Eisenstein series $E(s,\wt{g}, \phi,\wt{P}_2)$ coincide with those of its constant terms along $\wt{P}_2$, which is given by

$$E_{\wt{P}_2}(s, \wt{g}, \phi,\wt{P}_2)=\sum_{w\in \Omega_2} M(s,\wt{\chi}\boxtimes \sigma, w) \Phi_s(\wt{g}),$$
where $\Omega_2=\set{1,w_{121}}$. For any Weyl element $w$, any $f\in I_{\wt{P}_2}(s, \wt{\chi}\boxtimes \sigma)$ and Re$(s)\gg 0$,
$$M(s,\wt{\chi}\boxtimes \sigma, w)f(\wt{g})=\int_{N_2^w(\A_k)}f(w^{-1}n\wt{g})dn$$
converges absolutely. Similar as in proceeding section, we have $N_2^w=U\cap wN_2^\prime w^{-1}$, where $N_2^\prime$ is the unipotent subgroup of $U$ opposite to $N_2$. We also note that here $I_{\wt{P}_2}(s, \wt{\chi}\boxtimes \sigma)$ is the representation space of $\Mp_4(\A_k)$ generated by $\Phi_s$.

The tensor decomposition of $I_{\wt{P}_2}(s, \wt{\chi}\boxtimes \sigma)$ gives
$$M(s,\wt{\chi}\boxtimes \sigma, w)=\otimes A(s, \wt{\chi}_v\boxtimes \sigma_v, w), \text{ with } A(s, \wt{\chi}_v\boxtimes \sigma_v, w)f_v(\wt{g}_v)= \int_{N_2^w(k_v)}f_v(w^{-1}n\wt{g}_v)dn,$$
where $f=\otimes f_v$ and $f_v$ is the unique $k_v$-fixed function with normalization $f_v(e_v)=1$ for almost all $v$.

\subsection{Cuspidal Representation of $\Mp_2(\A_k)$}
We recall a few facts on the cuspidal representations $\sigma$ of $\Mp_2(\A_k)$. There are mainly two types to consider.

Let $\psi=\prod_v \psi_v$ be the fixed non-trivial additive character of $k\backslash \A_k$ given before for the Weil factor consideration. Then we have the decomposition
$$L_\text{cusp}^2(\Mp_2(\A_k))=(\text{ETF})\oplus \bigoplus_{\pi}L_\pi^2,$$
as $\pi$ ranges over all cuspidal representations of $\PGL_2(\A_2)$.

We give a brief explanation of the space $\text{ETF}$, which can be decomposed as
$$\text{ETF}=\bigoplus_{\eta}\text{ETF}_\eta=\bigoplus_{\eta}\Big(\bigoplus_S\omega_{\psi,\eta,S}\Big).$$
The $\eta$ is taken over all quadratic Hecke characters of $k^\times \backslash \A_k^\times$, and $S$ over finite set of positive even cardinality. Locally, for any place $v$ of $k$, consider the even and odd Weil representations $\omega^+_{\psi_v,\eta_v}$ and $\omega^-_{\psi_v,\eta_v}$ (cf. \cite{Kud}). Let $\wt{\chi}_{\psi_v}$ be the genuine character of $\wt{T}(k_v)$ attached to the Weil index, then the even Weil representation $\omega^+_{\psi_v,\eta_v}$
lies in the short exact sequence
$$0\to \omega^+_{\psi_v,\eta_v} \to \Ind_{\wt{B}_o}^{\Mp_2} \wt{\chi}_{\psi_v}(\eta_v |\ |^{-1/2}) \to st_{\psi_v,\eta_v} \to 0. $$

On the other hand, one knows that the odd Weil representations $\omega^-_{\psi_v,\eta_v}$ are supercuspidal. Now if $S$ is any nonempty finite set of places of $k$ with $|S|$ even, then the representation
$$\omega_{\psi,\eta,S}=\bigotimes_{v\in S}\omega^-_{\psi_v,\eta_v} \otimes \bigotimes_{v\notin S}\omega^+_{\psi_v,\eta_v}$$
is a cuspidal representation of $\Mp_2(\A_k)$. The representation space (ETF) is generated by all such $\omega_{\psi,\eta,S}$.

The complement of ETF is exhausted by $L_\pi^2$, which is given by theta correspondence with a cuspidal representation $\pi$ on $\PGL_2$ as input. The correspondence equates various arithmetic invariants, $L$-functions and $\gamma$ factors etc. For details, we refer to \cite{GaS}.

\subsection{Analysis of Local Intertwining Maps}
As in the Siegel case, we analyze both the local intertwining operator $A(s, \wt{\chi}_v\boxtimes \sigma_v,w_{121})$ for the almost all unramified places $v$ and the other finitely many exceptions.

We assume first $v$ is such that all $\psi_v, \chi_v$ and $\sigma_v$ are unramified, and $\eta_v$ as well in the ETF case. That is, we have $\sigma_v\hookrightarrow \Ind_{\wt{B}_o}^{\Mp_2}\wt{\mu}_v$ for an unramified character $\mu_v$ of $\GL_1(k_v)$.
By applying the Gindikin-Karpelevich formula in Proposition \ref{prop: main}, it follows
$$M(s,\wt{\chi}\boxtimes\sigma, w_{121})f=\frac{L^S(s,\chi\times \sigma)}{L^S(s+1,\chi\times \sigma)}\frac{L^S(2s,\chi^2)}{L^S(2s+1,\chi^2)}  \bigotimes_{v\notin S} f'_v \otimes \bigotimes_{v\in S} A(s, \wt{\chi}_v\boxtimes \sigma_v, w_{121}) f_v.$$

The Rankin-Selberg product $L(s,\chi\times \sigma)$, or more precisely the local analog, is given as in \cite{Szp} sect. 7. In order to determine the pole of $M(s,\wt{\chi}\boxtimes\sigma, w_{121})$, we need to analyze the partial $L$-functions $L^S$ and the local operators.

\subsection{The ETF Case}

First we analyze $L^S(s,\chi \times \sigma)$. Let $\sigma=\omega_{\psi,\eta,S_\sigma}\in $ ETF for some $S_\sigma$ of even cardinality. Without loss of generality we assume $S_\sigma\subseteq S$. The partial $L$-function becomes $L^S(s,\chi \times \sigma)=L^S(s-1/2,\chi\eta)L^S(s+1/2,\chi\eta)$ since locally this holds by the description of even Weil representation. The operator $M(s,\wt{\chi}\boxtimes\sigma, w_{121})$ can thus be simplified as

$$M(s,\wt{\chi}\boxtimes\sigma, w_{121})f=\frac{L^S(s-1/2,\chi\eta)}{L^S(s+3/2,\chi\eta)}\frac{L^S(2s,\chi^2)}{L^S(2s+1,\chi^2)}  \bigotimes_{v\notin S} f'_v \otimes \bigotimes_{v\in S} A(s, \wt{\chi}_v\boxtimes \sigma_v, w_{121}) f_v.$$

As before, we normalize the intertwining operator by giving

\begin{align}\label{ETF: normalizing}
r(s, \wt{\chi}_v\boxtimes \sigma_v, w_{121}) &=\frac{L(s,\chi_v\times \sigma_v)}{L(s+1,\chi_v\times \sigma_v)\varepsilon(s,\chi_v\times \sigma_v)}\frac{L(2s,\chi_v^2)}{L(2s+1,\chi_v^2)\varepsilon(2s,\chi_v^2)},\\
A(s, \wt{\chi}_v\boxtimes \sigma_v, w_{121})&=r(s, \wt{\chi}_v\boxtimes \sigma_v, w_{121}) R(s, \wt{\chi}_v\boxtimes \sigma_v, w_{121}). \notag
\end{align}

Regarding the holomorphicity of the normalized intertwining operator, we have

\begin{lm}\label{lm: ETF lm}
For $v\in S$, $R(s, \wt{\chi}_v\boxtimes \sigma_v,w_{121})$ can be continued to a holomorphic function for Re$(s)\ge 0$.
\end{lm}
\begin{proof} In all cases below, the denominator of the normalizing factor $r(s, \wt{\chi}_v\boxtimes \sigma_v,w_{121})$ has no zero or pole for Re$(s)\ge 0$, therefore we can neglect it for our purpose of proving the lemma. It suffices to compare the poles of $L(s,\chi_v\times \sigma_v)L(2s,\chi_v^2)$ and $A(s, \wt{\chi}_v\boxtimes \sigma_v,w_{121})$.

If $v\in S_\sigma$, $\sigma_v$ is supercuspidal and then we are in the tempered case. The local factor $L(s,\chi_v\times \sigma_v)$ has no poles and the operator $A(s, \wt{\chi}_v\boxtimes \sigma_v,w_{121})$ is holomorphic for Re$s\ge 0$.

Suppose $v\in S\backslash S_\sigma$, then $\sigma_v$ is the even Weil representation $\omega^+_{\psi_v,\eta_v}$. In this case, $\Ind_{\wt{P}_2}(s, \wt{\chi}_v\boxtimes \sigma_v)\hookrightarrow \Ind_{\wt{B}}\wt{\chi}_v|\ |^s\boxtimes \wt{\eta}_v|\ |^{-1/2}$ and the intertwining operator $A(s\beta_1,\wt{\chi}_v\boxtimes \wt{\eta}_v|\ |^{-1/2},w_{121})$ for the latter factorizes as
\begin{diagram}
I_{\wt{B}}(\wt{\chi}_v|\ |^s\boxtimes \wt{\eta}_v|\ |^{-1/2}) &\rTo^{A(w_1)}
&I_{\wt{B}}(\wt{\eta}_v|\ |^{-1/2} \boxtimes \wt{\chi}_v|\ |^s) &\rTo^{A(w_2)}
&I_{\wt{B}}(\wt{\eta}_v|\ |^{-1/2} \boxtimes \wt{\chi}_v^{-1}|\ |^{-s}) &\rTo^{A(w_1)}
&I_{\wt{B}}(\wt{\chi}_v^{-1}|\ |^{-s} \boxtimes \wt{\eta}_v|\ |^{-1/2}).
\end{diagram}

For Re$s\ge 0$, the first operator is holomorphic while the second has pole at $s=0$ when $\chi_v^2=\mbf{1}$. Also the last operator $A(w_1)$ gives a pole at $s=1/2$ when $\chi_v=\eta_v$.

The factor $L(s,\chi_v\times \sigma_v)$ defined as in \cite{Szp} sect. 7 is given by
$$L(s+1/2,\chi_v\eta_v)L(s-1/2,\chi_v\eta_v),$$
up to a sign. When $v$ is non-archimedean, we see already that the only possible pole of $L(s,\chi_v\times \sigma_v)L(2s,\chi_v^2)$ agrees with that of $A(s, \wt{\chi}_v\boxtimes \sigma_v,w_{121})$ for $s$ nonnegative.

If $v$ is real, then both $\chi_v, \eta_v$ are either trivial or the sign function. In this case the $L$-factor at real infinite place is given by
\begin{equation*}
L_\infty(s,\chi_v\times \sigma_v)=
\begin{cases}
\pi^{-s}\Gamma(\frac{s-1/2}{2})\Gamma(\frac{s+1/2}{2}) \text{ if } \chi_v\eta_v=\mbf{1}, \\
\pi^{-s-1}\Gamma(\frac{s+1/2}{2})\Gamma(\frac{s+3/2}{2}) \text{ if } \chi_v\eta_v=\text{sgn}.
\end{cases}
\end{equation*}
The only possible pole is at $s=1/2$ for $\chi_v=\eta_v=\text{sgn}$ . The complex case follows similarly. Therefore we see that in all cases, the poles of $L(s,\chi_v\times \sigma_v)L(2s,\chi_v^2)$ and $A(s, \wt{\chi}_v\boxtimes \sigma_v,w_{121})$ cancel each other and hence the lemma is proved.
\end{proof}

Now we are to consider the possible poles in different cases. There are three cases to consider

\begin{itemize}
\item[(1)] $\chi=\eta, s=3/2$;

\item[(2)] $\chi=\eta, s=1/2$;

\item[(3)] $\chi\ne \eta, \chi^2=\mbf{1}, s=1/2$.
\end{itemize}

\cu{First, $s=3/2, \chi=\eta.$} The intertwining operator $M(s,\wt{\chi}\boxtimes \sigma, w_{121})$ has a simple pole at $s=3/2$. We have the irreducibility of the image of intertwining operator, as in
\begin{lm}
For $s=3/2, \chi=\eta$, the image of $R(3/2, \wt{\chi}_v\boxtimes \sigma_v,w_{121})$ is irreducible and nonzero for all place $v$ of $k$.
\end{lm}
\begin{proof} The original operator $A(s, \wt{\chi}_v\boxtimes \sigma_v,w_{121})$ is holomorphic at $s=3/2$, and thus its image is equal to that of $R(s, \wt{\chi}_v\boxtimes \sigma_v,w_{121})$. We will show the irreducibility of image for the unnormalized operator.

If $\sigma_v$ is odd Weil representation, then it is tempered and the lemma follows from Langlands classification theorem. Now assume $\sigma_v=\omega_{\psi_v,\eta_v}$ is the even Weil representation, or equivalently as the image of the rank one intertwining operator given by the nontrivial Weyl element:
\begin{diagram}
I_{\wt{B}_o}^{\Mp_2}(\wt{\eta}_v |\ |^{1/2}) &\rTo &I_{\wt{B}_o}^{\Mp_2}(\wt{\eta}_v |\ |^{-1/2}).
\end{diagram}
The image of $A(s\beta_1, \wt{\chi}_v\boxtimes \sigma_v,w_{121})$ in this case is the image of the following composition $A(s\beta_1+1/4\alpha_2, \wt{\chi}_v\boxtimes \wt{\eta}_v, w_{121}w_2)$:
\begin{diagram}
I_{\wt{B}}(\wt{\chi}_v|\ |^s\boxtimes \wt{\eta}_v|\ |^{1/2}) &\rTo{A(w_2)} &I_{\wt{B}}(\wt{\chi}_v|\ |^s\boxtimes \wt{\eta}_v|\ |^{-1/2}) &\rTo{A(w_{121})}
I_{\wt{B}}(\wt{\chi}_v|\ |^{-s}\boxtimes \wt{\eta}_v|\ |^{-1/2})\\
 &\rdOnto & \uInto \\
 & & I_{\wt{P}_2}(\wt{\chi}_v|\ |^s\boxtimes \sigma_v).
\end{diagram}
For $s=3/2$, $s\beta_1+1/4\alpha_2$ lies in the positive Weyl chamber, so by Langlands classification, the image of $A(s\beta_1+1/4\alpha_2, \wt{\chi}_v\boxtimes \wt{\eta}_v, w_{121}w_2)$ is irreducible and the proof is completed.
\end{proof}

We see the image of $R(3/2, \wt{\chi}_v\boxtimes \sigma_v,w_{121})$ is the unique Langlands quotient with appropriate data. We denote it by $J_{\text{ETF}}(3/2,\wt{\chi}_v\boxtimes \sigma_v)$ and write
$$J_\text{ETF}(3/2,\wt{\chi}\boxtimes \sigma)=\bigotimes_v J_{\text{ETF}}(3/2,\wt{\chi}_v\boxtimes \sigma_v).$$

\cu{Second, $s=1/2, \chi=\eta.$} We combine the partial $L^S(s-1/2,\chi\eta) \prod_{v\in S\backslash S_\sigma}L_v(s-1/2,\chi\eta)$ from the coefficient of normalized intertwining operator. We see that it has at least a simple zero at $s=1/2$, since $S_\sigma$ is nonempty of even cardinality. Therefore since $L^S(2s,\chi^2)$ has a simple pole at $s=1/2$, we see the intertwining operator can not have any pole.

\cu{Third, $s=1/2, \chi\ne \eta, \chi^2=\mbf{1}$.} In this case, the operator $M(s,\wt{\chi}\boxtimes \sigma, w_{121})$ has a pole if and only if $L^{S_\sigma}(s-1/2,\chi\eta)$ is nonvanishing, since we have for the global $L$-function $L(s-1/2,\chi\eta)\ne 0$ for $s=1/2$. Equivalently we must have
$$\chi_v\ne \eta_v \text{ for all } v\in S_\sigma.$$

Meanwhile, under this condition
\begin{lm}
The image of $R(1/2, \wt{\chi}_v\boxtimes \sigma_v,w_{121})$ is irreducible and nonzero for all $v$.
\end{lm}
\begin{proof} The argument is essential the same as previous lemma. It suffices to consider $\sigma_v$ being the even Weil representation case. We may assume $\chi_v=\eta_v$, otherwise the operator $A(s, \wt{\chi}_v\boxtimes \sigma_v,w_{121})$ is holomorphic at $s=1/2$ and the argument goes line by line as the previous lemma.

If $\chi_v=\eta_v$, then the operator $A(s, \wt{\chi}_v\boxtimes \sigma_v,w_{121})$ has a simple pole at $s=1/2$ and the image of $R(1/2, \wt{\chi}_v\boxtimes \sigma_v,w_{121})$ is equal to the image of $\Res_{s=1/2}A(s, \wt{\chi}_v\boxtimes \sigma_v,w_{121})$ (the right diagonal map below). We have the commutative diagram

\begin{diagram}
I_{\wt{B}}(\wt{\chi}_v|\ |^s\boxtimes \wt{\eta}_v|\ |^{1/2}) &\rTo{A(w_2)} & I_{\wt{B}}(\wt{\chi}_v|\ |^s\boxtimes \wt{\eta}_v|\ |^{-1/2}) &\rTo{\Res_{s=1/2}A(w_{121})}
& I_{\wt{B}}(\wt{\chi}_v|\ |^{-s}\boxtimes \wt{\eta}_v|\ |^{-1/2}).\\
&\rdOnto & \uInto  & \ruTo_{\Res_{s=1/2}A(w_{121})}\\
& & I_{\wt{P}_2}(\wt{\chi}_v|\ |^s\boxtimes \sigma_v)
\end{diagram}

It follows from the diagram that it suffices to consider the image of the top composition map $\big(\Res_{s=1/2}A(s\beta_1-1/4\alpha_2, \wt{\chi}_v\boxtimes \wt{\eta}_v, w_{121})\big)\cdot A(s\beta_1+1/4\alpha_2, \wt{\chi}_v\boxtimes \wt{\eta}_v, w_2)$, which is equal to
$$\big(\Res_{s=1/2}A(-\alpha_4/4-s\alpha_2/2, \wt{\eta}_v\boxtimes \wt{\chi}_v, w_1)\big)\cdot A(s\beta_1+1/4\alpha_2, \wt{\chi}_v\boxtimes \wt{\eta}_v, w_{212}).$$

When $s=1/2$, $s\beta_1+1/4\alpha_2=\alpha_3/2$. By induction in stages to the Siegel parabolic subgroup $\wt{P}_1$, $I_{\wt{B}}(\alpha_3/2, \wt{\chi}_v \boxtimes \wt{\eta}_v)=I_{\wt{P}_1}(\alpha_3/2, \Ind\wt{\chi}_v\boxtimes \wt{\eta}_v)$. So the image of $A(\alpha_3/2, \wt{\chi}_v\boxtimes \wt{\eta}_v, w_{212})$ is irreducible by Langlands classification theorem. Moreover, reduction to a simple rank one computation shows that the map
$$\Res_{s=1/2}A(-\alpha_4/4-s\alpha_2/2, \wt{\eta}_v\boxtimes \wt{\chi}_v, w_1): I_{\wt{B}}(\wt{\eta}_v|\ |^{-1/2}\boxtimes \wt{\chi}_v|\ |^{-s}) \longrightarrow I_{\wt{B}}(\wt{\chi}_v|\ |^{-s}\boxtimes \wt{\eta}_v|\ |^{-1/2})$$ 
is a scaling map by a nonzero constant and thus an isomorphism. The proof of lemma is now completed.
\end{proof}

Based on this lemma, the image of $R(1/2, \wt{\chi}_v\boxtimes \sigma_v,w_{121})$ is a Langlands quotient, which we denote by $J_{\text{ETF}}(1/2,\wt{\chi}_v\boxtimes \sigma_v)$. Define
$$J_\text{ETF}(1/2,\wt{\chi}\boxtimes \sigma)=\bigotimes J_{\text{ETF}}(1/2,\wt{\chi}_v\boxtimes \sigma_v).$$
We see both $J_\text{ETF}(3/2,\wt{\chi}\boxtimes \sigma)$ and $J_\text{ETF}(1/2,\wt{\chi}\boxtimes \sigma)$ contribute to $L^2_\text{d}(\wt{P}_2)$.

\subsection{The $L^2_\pi$ case}

In this case, the local representation $\sigma_v$ is the obtained from the local component $\pi_v$ of a cuspidal representation $\pi$ on $\PGL_2(\A_k)$ by Shimura correspondence. Therefore $L^S(s,\chi\times \sigma)=L^S(s,\chi\times \pi)$ and the latter is known to be entire over $\C$. The operator $M(s,\wt{\chi}\boxtimes \sigma, w_{121})$ has a pole at $s=1/2$ provided $\chi^2=\mbf{1}, L(1/2,\chi\times \pi)\ne 0$ and that the normalized local operator for $v\in S$ is holomorphic and nonvanishing at $\re (s)\ge 0$. More precisely, we use exactly the same formula (\ref{ETF: normalizing})  in the ETF case for the normalization. Then

\begin{lm}
For $v\in S$, the operator $R(s,\wt{\chi}_v\boxtimes \sigma_v,w_{121})$ can be continued to be a holomorphic function for Re$(s)\ge0$.
\end{lm}
\begin{proof} The proof follows from similar computation as in Lemma \ref{lm: ETF lm}.

In this case $\sigma_v$ is the theta lift of an irreducible representation $\pi_v$ on $\PGL_2(k_v)$ or $\text{PD}_v^\times$. We know that $\sigma_v$ is tempered if and only if the same holds for $\pi_v$, and also $L(s,\chi_v\times \sigma_v)=L(s,\chi_v\times \pi_v)$. We refer to \cite{GaS} for more properties for the theta lift and its consequences.

In view of this, the lemma holds in the tempered case. Now assume $\sigma_v$ is not tempered, then we have $\pi_v=\Ind_{B_o}\nu|\ |^a\boxtimes \nu|\ |^{-a}$ with $\nu^2=\mbf{1}$ and $a\in (0,1/9)$. Hence we have $\sigma_v=\Ind_{\wt{B}_o}\wt{\nu}|\ |^a$. The factorization of the intertwining operator $A(s\beta_1,\wt{\chi}_v\boxtimes \wt{\nu}|\ |^a,w_{121})$ gives
\begin{diagram}
I_{\wt{B}}(\wt{\chi}_v|\ |^s\boxtimes \wt{\nu}|\ |^a) &\rTo^{A(w_1)}
&I_{\wt{B}}(\wt{\nu}|\ |^a \boxtimes \wt{\chi}_v|\ |^s) &\rTo^{A(w_2)}
&I_{\wt{B}}(\wt{\nu}|\ |^a \boxtimes \wt{\chi}_v^{-1}|\ |^{-s}) &\rTo^{A(w_1)}
&I_{\wt{B}}(\wt{\chi}_v^{-1}|\ |^{-s} \boxtimes \wt{\nu}|\ |^a).
\end{diagram}
For Re$(s)\ge 0$, third operator is holomorphic, while the first operator $A(w_1)$ gives a pole at $s=a$ when $\chi_v=\nu$, the second operator at $s=0$ when $\chi_v^2=\mbf{1}$. On the other side, the factor $L(s,\chi_v\times \sigma_v)$ is given by $L(s+a,\chi_v\nu)L(s-a,\chi_v\nu)$, whose pole is at $s=a$ for $\chi_v=\nu$ as well. Therefore the potential poles of the normalizing fact $r(s,\wt{\chi}_v\boxtimes \sigma_v,w_{121})$ and original operator $A(s,\wt{\chi}_v\boxtimes \sigma_v,w_{121})$ cancel each other.

\end{proof}

The image of $R(1/2,\wt{\chi}_v\boxtimes \sigma_v,w_{121})$, by the same argument as in Lemma \ref{lm: image irred}, is a Langlands quotient with appropriate data. We denote the image by $J_\pi(1/2, \wt{\chi}_v\boxtimes \sigma_v)$ and let
$$J_\pi(1/2, \wt{\chi}\boxtimes \sigma)=\bigotimes J_\pi(1/2, \wt{\chi}_v\boxtimes \sigma_v),$$
which contributes to the residual spectrum $L^2_\text{d}(\wt{P}_2)$.

In conclusion, we have proved the following

\begin{thm}\label{thm: non-siegel}
Let $\msc{P}_1, \msc{P}_2 \subseteq \msc{A}_\text{cusp}(\GL_1(\A_k)) \times \msc{A}_\text{cusp}(\Mp_2(\A_k))$ be defined as the collections of pairs:
$$\msc{P}_1=\set{(\chi, \sigma): \chi=\eta, \sigma\in \text{ETF}_\eta}, \msc{P}_2=\set{(\chi, \sigma): \chi^2=\mbf{1}, \sigma\in \text{ETF}_\eta, \chi_v\ne \eta_v \text{ if } v\in S_\sigma},$$
where it is always assumed that $\eta^2=\mbf{1}$. Also we let
\begin{align*}
\msc{P}_3 &= \set{(\chi, \pi): \chi^2=\mbf{1}, L(1/2,\chi\times\pi)\ne 0} \\
&\subseteq \msc{A}_\text{cusp}(\GL_1(\A_k)) \times \msc{A}_\text{cusp}(\PGL_2(\A_k)).
\end{align*}

Then the residual spectrum of $\Mp_4(\A_k)$ along $\wt{P}_2$ is given by
$$L^2_d(P_2)=\bigoplus_{(\chi,\sigma)\in\msc{P}_1} J_\text{ETF}(3/2,\wt{\chi}\boxtimes \sigma) \oplus \bigoplus_{(\chi,\sigma)\in\msc{P}_2} J_\text{ETF}(1/2,\wt{\chi}\boxtimes \sigma) \oplus \bigoplus_{(\chi,\pi)\in \msc{P}_3} \bigoplus_{\sigma\in L^2_\pi} J_\pi(1/2,\wt{\chi}\boxtimes \sigma).$$
\end{thm}

\section{Decomposition for the Borel subgroup}
Given two unitary Hecke character $\chi, \mu$ of $\GL_1(\A_k)$, we define a genuine character $\wt{\chi}\boxtimes \wt{\mu}$ of $\wt{T}(\A_k)$ as before. Consider the space $\Ind_{\wt{B}}^{\Mp_4}\wt{\chi}\boxtimes \wt{\mu}$ of functions $\phi$ on $\Mp_4(\A_k)$ satisfying $\phi(n\wt{t}\wt{g})=(\wt{\chi}\boxtimes \wt{\mu})(\wt{t})\phi(g)$ for any $n\in U(\A_k), \wt{t}\in \wt{T}(\A_k)$ and $\wt{g}\in \Mp_4(\A_k)$. For each $\Lambda \in \mfr{a}^*_\C=\C \beta_1\oplus \C\beta_2$, the representation of $\Mp_4(\A_k)$ on the space of functions of the form
$$\Phi_\Lambda: \quad \wt{g}\mapsto \phi(\wt{g})\exp \angb{\Lambda+\rho_B}{H_{\wt{B}}(\wt{g})},$$
is equivalent to $I_{\wt{B}}(\Lambda, \wt{\chi}\boxtimes \wt{\mu})=\text{Ind}_{\wt{B}}^{\Mp_4}\otimes \exp\angb{\Lambda}{H_{\wt{B}}()}$, where $\rho_B$ is the half-sum of positive roots, i.e. $\rho_B=\beta_1+ \beta_2$. We form the Eisenstein series as before
$$E(\Lambda, \wt{g}, \phi,\wt{B})=\sum_{\gamma\in B(k)\backslash \Mp_4(k)}\Phi_\Lambda(\gamma \wt{g}).$$

Since $\Phi_\Lambda(\wt{g})=\phi(\wt{g}) \exp \angb{\Lambda+\rho_B}{H_{\wt{B}}(\wt{g})}$, we may sometimes write $E(\Lambda, \wt{g}, \Phi_\Lambda,\wt{B})$ for $E(\Lambda, \wt{g}, \phi,\wt{B})$ to emphasize the variation of functions in the representation space $I_{\wt{B}}(\Lambda,\wt{\chi}\boxtimes \wt{\mu})$. We know that the Eisenstein series satisfies the functional equation
$$E(\Lambda, \wt{g}, \Phi_\Lambda, \wt{B})=E(w\Lambda, \wt{g}, M(\Lambda,\wt{\chi}\boxtimes \wt{\mu},w)\Phi_\Lambda, \wt{B})$$
for any Weyl group element $w$. Moreover, the constant term of $E(\Lambda, \wt{g}, \Phi_\Lambda,\wt{B})$ along $\wt{B}$ is given by
$$E_{\wt{B}}(\Lambda, \wt{g}, \Phi_\Lambda,\wt{B})=\sum_{w\in \Omega} M(\Lambda,\wt{\chi}\boxtimes\wt{\mu}, w) \Phi_\Lambda(\wt{g}),$$
where $\Omega$ is the full Weyl group. In general, for $w\in \Omega$ and $f\in I_{\wt{B}}(\Lambda, \wt{\chi}\boxtimes \wt{\mu})$ and sufficiently regular $\Lambda$,
$$M(\Lambda,\wt{\chi}\boxtimes\wt{\mu}, w)f(\wt{g})=\int_{U^w(\A_k)}f(w^{-1}n\wt{g})dn.$$

The symbol $U^w$ represents $U\cap wU^\prime w^{-1}$. It is known that
$$M(\Lambda,\wt{\chi}\boxtimes\wt{\mu}, w)=\otimes A(\Lambda, \wt{\chi}_v\boxtimes\wt{\mu}_v, w), \text{ with } A(\Lambda, \wt{\chi}_v\boxtimes \wt{\mu}_v, w)f_v(\wt{g}_v)= \int_{U^w(k_v)}f_v(w^{-1}n\wt{g}_v)dn,$$
where $f=\otimes f_v$ and $f_v$ is the unique $\Sp_4(\mfr{O}_v)$-fixed function with normalization $f_v(e_v)=1$ for almost all $v$.

For almost all places, Proposition \ref{prop: main} gives the coefficients attached to the local operator $A(\Lambda, \wt{\chi}_v\boxtimes \wt{\mu}_v, w)$. For convenience, for $w\in \Omega$ we normalize $M(\Lambda,\wt{\chi}\boxtimes\wt{\mu}, w)$ as follows. 

First we introduce some notations to be used in the sequel. Consider any Hecke $L$-function attached to a grossencharacter $\theta=\otimes_v \theta_v$ of $k$, and the additive character $\psi=\otimes_v \psi_v$ of $k\backslash \A_k$ given before, we let $\varepsilon(s,\theta)=\prod_v\varepsilon(s,\theta_v,\psi_v)$ denote the usual $\varepsilon$ factor. For simplicity, we may write $\varepsilon(s,\theta_v)=\varepsilon(s,\theta_v,\psi_v)$. Then the functional equation
$$L(s,\theta)=\varepsilon(s,\theta)L(1-s,\theta^{-1})$$
holds. We write
$$\zeta(s)=L(s,\mbf{1})=\frac{a_{-1}}{s-1}+a_0 +a_1 (s-1) + ..., \text{ with } a_{-1}=\Res_{s=1}\zeta(s).$$

Now the normalization is given by
\begin{align*}
r(\Lambda, \wt{\chi}_v\boxtimes \wt{\mu}_v, w) = & \prod_{\substack{\beta> 0, w\beta< 0 \\ \beta \text{ short}}} \frac{L(\angb{\Lambda}{\beta^\vee}, (\chi_v\boxtimes\mu_v) \beta^\vee)}{L(1+\angb{\Lambda}{\beta^\vee}, (\chi_v\boxtimes\mu_v) \beta^\vee)\varepsilon(\angb{\Lambda}{\beta^\vee}, (\chi_v\boxtimes\mu_v) \beta^\vee, \psi_v)} \\
& \cdot \prod_{\substack{\beta> 0, w\beta< 0 \\ \beta \text{ long}}} \frac{L(2\angb{\Lambda}{\beta^\vee}, ((\chi_v\boxtimes\mu_v) \beta^\vee)^2)}{L(1+2\angb{\Lambda}{\beta^\vee}, ((\chi_v\boxtimes\mu_v) \beta^\vee)^2)\varepsilon(2\angb{\Lambda}{\beta^\vee}, ((\chi_v\boxtimes\mu_v) \beta^\vee)^2, \psi_v)} \\
A(\Lambda, \wt{\chi}_v\boxtimes \wt{\mu}_v, w) &=r(\Lambda, \wt{\chi}_v\boxtimes \wt{\mu}_v, w)R(\Lambda, \wt{\chi}_v\boxtimes \wt{\mu}_v, w).
\end{align*}
Regarding the holomorphicity of $R(\Lambda, \wt{\chi}_v\boxtimes \wt{\mu}_v, w)$, we have

\begin{lm}
The normalized operator $R(\Lambda, \wt{\chi}_v\boxtimes \wt{\mu}_v, w)$ is holomorphic for Re$\angb{\Lambda}{\beta}\ne -1 \text{ or }-1/2$ when $\beta$ is short or long, respectively. Here $\beta$ is any positive root that appears in the expression for $R(\Lambda, \wt{\chi}_v\boxtimes \wt{\mu}_v, w)$ above.
\end{lm}

\begin{proof} It follows from the same reasoning as in Lemma \ref{lm: ETF lm}, or modified argument in \cite{Win} . To illustrate, we consider for example the case $w=w_{121}$ in more details and others follows similarly. In this case, parameterize $\Lambda=s\beta_1 + t\beta_2$, then $I_{\wt{B}}(\Lambda, \wt{\chi}_v\boxtimes \wt{\mu}_v)=I_{\wt{B}}(\wt{\chi}_v|\ |^{s+t}\boxtimes \wt{\mu}_v|\ |^t)$ and the operator $A(\Lambda, \wt{\chi}_v\boxtimes \wt{\mu}_v, w_{121})$ factorizes as
\begin{diagram}
I_{\wt{B}}(\wt{\chi}_v|\ |^{s+t}\boxtimes \wt{\mu}_v|\ |^t) &\rTo &I_{\wt{B}}(\wt{\mu}_v|\ |^t \boxtimes \wt{\chi}_v|\ |^{s+t}) &\rTo &I_{\wt{B}}(\wt{\mu}_v|\ |^t \boxtimes \wt{\chi}_v^{-1}|\ |^{-s-t}) &\rTo  &I_{\wt{B}}(\wt{\chi}_v^{-1}|\ |^{-s-t}\boxtimes \wt{\mu}_v|\ |^t).
\end{diagram}
Individually, the simple pole of each map in the chain occur at $s=0$ with $\chi_v=\mu_v$, $s+t=0$ with $\chi_v^2=\mbf{1}$ and $s+2t=0$ with $\chi_v=\mu_v^{-1}$. On the other hand, we assume $\Lambda$ is such that $\angb{\Lambda}{\alpha_1}>-1, \angb{\Lambda}{\alpha_3}>-1$ and $\angb{\Lambda}{\alpha_4}>-1/2$. Then the poles of $r(\Lambda, \wt{\chi}_v\boxtimes \wt{\mu}_v, w_{121})$ are among those of
\begin{align*}
&L(\angb{\Lambda}{\alpha_1^\vee}, (\chi_v\boxtimes\mu_v) \alpha_1^\vee)L(\angb{\Lambda}{\alpha_3^\vee}, (\chi_v\boxtimes\mu_v) \alpha_3^\vee)L(2\angb{\Lambda}{\alpha_4^\vee}, ((\chi_v\boxtimes\mu_v) \alpha_4^\vee)^2)\\
=&L(s,\chi_v\mu_v^{-1})L(s+2t,\chi_v\mu_v)L(2s+2t,\chi_v^2).
\end{align*}

It can be observed that the occurrence of the pole agrees with that of $A(\Lambda, \wt{\chi}_v\boxtimes \wt{\mu}_v, w_{121})$.
\end{proof}

Note also that the normalized operator sends an unramified vector to an unramified one in the codomain space. It satisfies similar cocycle relation $R(\Lambda, \wt{\chi}_v\boxtimes \wt{\mu}_v, w'w)=R(w\Lambda, w(\wt{\chi}_v\boxtimes \wt{\mu}_v), w') R(\Lambda, \wt{\chi}_v\boxtimes \wt{\mu}_v, w)$ for $w, w'$ any Weyl elements with the relation $l(ww')=l(w) + l(w')$. Here $l(w)$ is the total number of $w_1$ and $w_2$ that appear in the reduced decomposition of $w$.

We see that in order to determine the pole of $E_{\wt{B}}(\Lambda, \wt{g}, \Phi_\Lambda,\wt{B})$, it is necessary to consider each $M(\Lambda,\wt{\chi}\boxtimes\wt{\mu}, w)$, and therefore the several hyperplanes in the root space defined as
\begin{equation*}
S_i=
\begin{cases}
\set{\Lambda \in \mfr{a}^*_\C | \angb{\Lambda}{\alpha_i}=1} \text{ if } \alpha_i \text{ is short, } \\
\set{\Lambda \in \mfr{a}^*_\C | \angb{\Lambda}{\alpha_i}=1/2} \text{ if } \alpha_i \text{ is long, }
\end{cases}
\end{equation*}

\noindent where $i=1,2,3,4$. The hyperplanes are drawn in the figure.

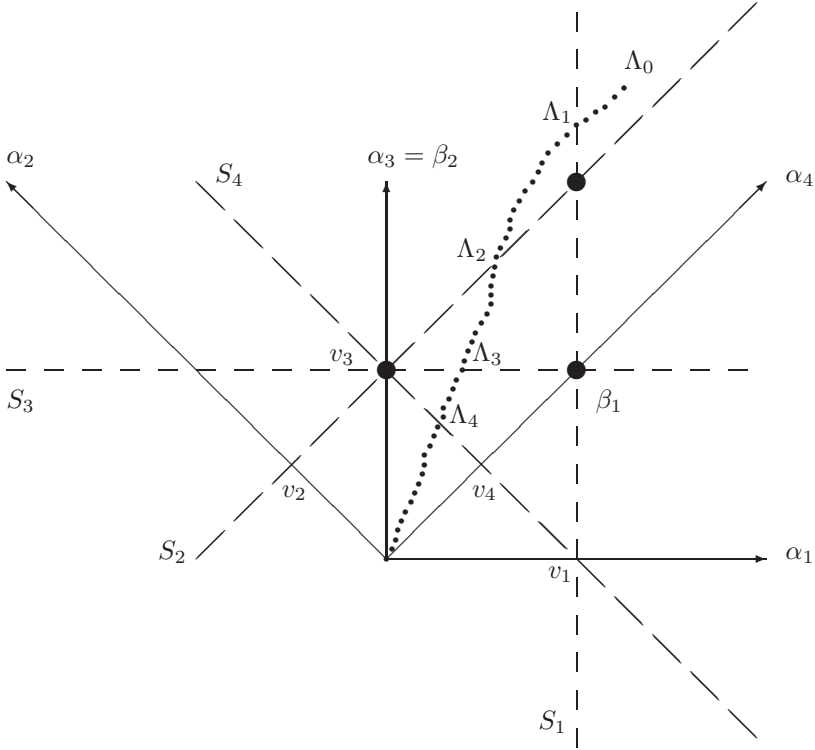
\setlength{\unitlength}{2.5cm}
\begin{picture}(4.7,4.7)(-0.3,-0.3)
\put(2,1){\vector(1,0){2}}
\put(2,1){\vector(1,1){2}}
\put(2,1){\vector(0,1){2}}
\put(2,1){\vector(-1,1){2}}
\multiput(3,0)(0,0.2){20}{\line(0,1){0.1}}
\multiput(0,2)(0.2,0){20}{\line(1,0){0.1}}
\multiput(1,1)(0.2,0.2){15}{\line(1,1){0.15}}
\multiput(1,3)(0.2,-0.2){15}{\line(1,-1){0.15}}
\put(4.1,1){$\alpha_1$}
\put(0,3.1){$\alpha_2$}
\put(1.9,3.1){$\alpha_3=\beta_2$}
\put(4.1,3){$\alpha_4$}
\put(3.1,1.8){$\beta_1$}
\put(3,2){\circle*{0.1}}
\put(3,3){\circle*{0.1}}
\put(2,2){\circle*{0.1}}
\put(2.8,0.1){$S_1$}
\put(0.8,1){$S_2$}
\put(0,1.8){$S_3$}
\put(1.1,3){$S_4$}

\put(3.25,3.6){$\Lambda_0$}
\put(3.25,3.5){\circle*{0.03}}
\put(3.20,3.45){\circle*{0.03}}
\put(3.15,3.40){\circle*{0.03}}
\put(3.10,3.37){\circle*{0.03}}
\put(3.05,3.33){\circle*{0.03}}
\put(3.00,3.30){\circle*{0.03}}
\put(2.82,3.33){$\Lambda_1$}
\put(2.95,3.25){\circle*{0.03}}
\put(2.90,3.20){\circle*{0.03}}
\put(2.85,3.15){\circle*{0.03}}
\put(2.82,3.10){\circle*{0.03}}

\put(2.80,3.05){\circle*{0.03}}
\put(2.77,3){\circle*{0.03}}
\put(2.73,2.95){\circle*{0.03}}
\put(2.70,2.90){\circle*{0.03}}
\put(2.67,2.85){\circle*{0.03}}
\put(2.65,2.80){\circle*{0.03}}
\put(2.65,2.75){\circle*{0.03}}
\put(2.63,2.70){\circle*{0.03}}
\put(2.60,2.65){\circle*{0.03}}
\put(2.58,2.60){\circle*{0.03}}
\put(2.37,2.60){$\Lambda_2$}

\put(2.57,2.55){\circle*{0.03}}
\put(2.56,2.50){\circle*{0.03}}
\put(2.55,2.45){\circle*{0.03}}
\put(2.55,2.40){\circle*{0.03}}
\put(2.55,2.35){\circle*{0.03}}
\put(2.53,2.30){\circle*{0.03}}
\put(2.50,2.25){\circle*{0.03}}
\put(2.47,2.20){\circle*{0.03}}
\put(2.45,2.15){\circle*{0.03}}
\put(2.43,2.10){\circle*{0.03}}

\put(2.41,2.05){\circle*{0.03}}
\put(2.40,2.00){\circle*{0.03}}
\put(2.45,2.05){$\Lambda_3$}
\put(2.37,1.95){\circle*{0.03}}
\put(2.35,1.90){\circle*{0.03}}
\put(2.33,1.85){\circle*{0.03}}
\put(2.30,1.80){\circle*{0.03}}
\put(2.30,1.75){\circle*{0.03}}
\put(2.27,1.70){\circle*{0.03}}
\put(2.33,1.72){$\Lambda_4$}
\put(2.25,1.65){\circle*{0.03}}
\put(2.23,1.60){\circle*{0.03}}

\put(2.20,1.55){\circle*{0.03}}
\put(2.20,1.50){\circle*{0.03}}
\put(2.19,1.45){\circle*{0.03}}
\put(2.17,1.40){\circle*{0.03}}
\put(2.15,1.35){\circle*{0.03}}
\put(2.13,1.30){\circle*{0.03}}
\put(2.10,1.25){\circle*{0.03}}
\put(2.08,1.20){\circle*{0.03}}
\put(2.06,1.15){\circle*{0.03}}
\put(2.05,1.10){\circle*{0.03}}
\put(2.03,1.05){\circle*{0.03}}
\put(2.00,1.00){\circle*{0.03}}

\put(2.85,0.9){$v_1$}
\put(2.45,1.33){$v_4$}
\put(1.45,1.33){$v_2$}
\put(1.7,2.05){$v_3$}

\end{picture}

We proceed first with a general discussion of $L^2_\text{d}(\wt{B})$. Consider any entire function $F$ on $\mfr{a}^*_\C$ of Paley-Wiener type such that $F(\Lambda)\in I_{\wt{B}}(\Lambda,\wt{\chi}\boxtimes \wt{\mu})$. According to Langlands theory, $L^2(\wt{B})$ is generated by
$$\Theta_F(\wt{g})=\frac{1}{(2\pi i)^2}\int_{\re \Lambda=\Lambda_0}E(\Lambda,\wt{g},F(\Lambda),\wt{B})d\Lambda,$$
for all such $F(\Lambda)$. Here $\Lambda_0$ satisfies $\angb{\Lambda_0-\rho_B}{\alpha^\vee}>0$ for all positive roots $\alpha$ (cf. \cite{MoW} II.1). In fact in our case for $\Mp_4$, we can infer from the figure that $\rho_B$ here can be replaced by $\alpha_1/2+\alpha_3$. In order to obtain the discrete spectrum $L^2_\text{d}(\wt{B})$, we have to deform the contour from $\re \Lambda=\Lambda_0$ to $\re \Lambda=0$. A priori, the poles of the intertwining operators $M(\Lambda,\wt{\chi}\boxtimes \wt{\mu},w)$ lie on $S_i$ given by real equations, and here we use the curly dotted line to represent the deformation process. The resulting integral at Re$\Lambda=0$,
$$\frac{1}{(2\pi i)^2}\int_{\re \Lambda=0}E(\Lambda,\wt{g},F(\Lambda),\wt{B})d\Lambda,$$
will contribute to the continuous spectrum of $L^2(\wt{B})$ of dimension 2. However, by the theory of several complex variables, such deformation will enable us to pick up first-order residues of the form
$$\frac{1}{2\pi i}\int_{\re \Lambda=\Lambda_i}\text{Res}_{S_i}E(\Lambda,\wt{g},F(\Lambda),\wt{B})d\Lambda,$$
where the variable $\Lambda$ now lies in $S_i$.

Now we can further deform the contour from Re$\Lambda=\Lambda_i$ to Re$\Lambda=v_i$ along the hyperplane $S_i$. Thus the integral at Re$\Lambda=v_i$,
$$\frac{1}{2\pi i}\int_{\re \Lambda=v_i}\text{Res}_{S_i}E(\Lambda,\wt{g},F(\Lambda),\wt{B})d\Lambda,$$
will give the continuous spectrum of $L^2(\wt{B})$ of dimension 1. The square integrable residues that arise from taking residues of the integrand $\text{Res}_{S_i}E(\Lambda,\wt{g},F(\Lambda),\wt{B})$ will span the discrete spectrum $L^2_\text{d}(\wt{B})$. 

From the figure, we see that one only needs to look at
$$\Res_{\beta_1+\frac{\beta_2}{2}} \Res_{S_1} E(\Lambda, \wt{g}, \Phi_\Lambda, \wt{B}), \quad \Res_{\beta_1} \Res_{S_1} E(\Lambda, \wt{g}, \Phi_\Lambda, \wt{B}),$$
$$\Res_\frac{\beta_2}{2} \Res_{S_2} E(\Lambda, \wt{g}, \Phi_\Lambda, \wt{B}), \quad \Res_\frac{\beta_2}{2} \Res_{S_3} E(\Lambda, \wt{g}, \Phi_\Lambda, \wt{B}).$$

To facilitate the computations of above residues, we first prove a useful lemma

\begin{lm}\label{lm: id}
For $\chi=\mu, \chi^2=\mbf{1}$, the equalities $w_1(\alpha_3)=\alpha_3, w_2(\alpha_4)=\alpha_4$ hold. Also the character $\wt{\chi}\boxtimes \wt{\mu}$ is invariant under both $w_1$ and $w_2$. Let $t\in \R$ be any real number, then the two intertwining operators
$$R(t\alpha_3, \wt{\chi}\boxtimes \wt{\mu}, w_1): I_{\wt{B}}(t\alpha_3,\wt{\chi}\boxtimes \wt{\mu}) \longrightarrow I_{\wt{B}}(w_1(t\alpha_4),w_1(\wt{\chi}\boxtimes \wt{\mu}))$$
$$R(t\alpha_4, \wt{\chi}\boxtimes \wt{\mu}, w_2): I_{\wt{B}}(t\alpha_4,\wt{\chi}\boxtimes \wt{\mu}) \longrightarrow I_{\wt{B}}(w_2(t\alpha_4),w_2(\wt{\chi}\boxtimes \wt{\mu}))$$
are both the identity maps.
\end{lm}
\begin{proof} By inducing in stages
$$I_{\wt{B}}(t\alpha_3,\wt{\chi}\boxtimes \wt{\mu})=I_{\wt{P}_1}(I_{\wt{B}_o}^{\wt{\GL}_2}(|\ |^t\wt{\chi}\boxtimes|\ |^t\wt{\mu})).$$
Since induction commutes with the intertwining operator, one is reduced to the rank one $\wt{\GL}_2$ case. However, the rank one intertwining operator is essentially a $\GL_2$ operator. That is, it suffices to show that
$R(\chi \boxtimes \chi,w): I_{B_o}^{\GL_2}(\chi \boxtimes \chi)\longrightarrow I_{\wt{B}_o}^{\GL_2}(\chi \boxtimes \mu)$
is the identity map, where $w$ is the nontrivial Weyl reflection and $R(\chi \boxtimes \chi,w)$ is the intertwining operator 
$$R(\chi\boxtimes \chi,w)=\lim_{s\to 0}R(s, \chi\boxtimes \chi,w) \text{ with } 
M(s, \chi\boxtimes \chi,w)=\frac{\zeta(s)}{\zeta(s+1)}R(s, \chi\boxtimes \chi,w).$$

However, we know $\lim_{s\to 0}M(\chi\boxtimes \chi,w)=-\text{identity}$ (cf. \cite{KeS} prop. 6.3), which gives $R(\chi\boxtimes \chi,w)=\text{identity}$.

For the second case, we have $$I_{\wt{B}}(t\alpha_4,\wt{\chi}\boxtimes \wt{\mu})=I_{\wt{P}_2}(|\ |^{2t}\wt{\chi}\boxtimes I_{\wt{B}_o}^{\Mp_2}\wt{\mu}).$$
Consider the un-normalized operator
$$M(\wt{\mu}, w): \Ind_{\wt{B}_o}^{\Mp_2}(\wt{\mu}) \longrightarrow \Ind_{\wt{B}_o}^{\Mp_2}(w\wt{\mu}),$$
which appears in the constant term of the rank one $\Mp_2$ Eisenstein series $E_{\wt{B}_o}(s,\wt{g},\phi)$ evaluated at $s=0$. Here $\phi\in I_{\wt{B}_o}(\wt{\mu})$ and $\wt{g}\in \Mp_2$. By \cite{Szp} sect. 9, the non-constant term Fourier coefficient of $E(s,\wt{g},\phi)$ vanishes at $s=0$, which gives $E(0,\wt{g},\phi)=E_{\wt{B}_o}(0,\wt{g},\phi)$, the constant term. .

Suppose the automorphic form $E(0,\wt{g},\phi)$ is nonzero. Clearly, on the subdomain $\wt{B}_o(\A_k)$ it is invariant under the unipotent subgroup $U(\A_k)$, i.e. $E(0,\wt{b} n,\phi)=E(0,\wt{b},\phi)$ for all $\wt{b}\in \wt{B}_o(\A_k)$ and $n\in U(\A_k)$. Note that it is also invariant under $B(k)$. By strong approximation theorem we know that $B_o(k)\backslash \wt{B}_o(\A_k)$ is dense in $\Mp_2(k)\backslash \Mp_2(\A_k)$. Therefore $E(0,\wt{g},\phi)$ is invariant under $U(\A_k)$ on the whole domain, which we assume belongs to $\bigotimes_v\pi_v$.  Then for almost all $v$, $\pi_v$ is an infinite dimensional representation invariant under $U(k_v)$. This violates the Howe-Moore vanishing property at infinity of the matrix coefficients of $\pi_v$ (cf. \cite{HoM}), unless $E(0,\wt{g},\phi)=0$.

It follows that $E(0,\wt{g},\phi)=0$ and
$$M(\wt{\mu}, w)=-\text{identity}=\lim_{s\to 0}\frac{\zeta(2s)}{\zeta(2s+1)}R(\wt{\mu}, w)=-R(\wt{\mu}, w).$$
We also note that this argument in fact works for both cases above. The proof is completed.
\end{proof}

The computation of iterated residues relies on the coordinate system that we use. In what follows, we will set as the basic one the coordinates
\begin{align} \label{cor_xy}
\Lambda=x\alpha_1/2 + y\alpha_3/2.
\end{align}
Therefore, the intertwining operators can be viewed as functions of $x$ and $y$. In view of above lemma, one has immediately

\begin{cor}\label{cor: partial deriv}
With coordinates $\Lambda=x\alpha_1/2 + y\alpha_3/2$. Assume $\chi=\mu, \chi^2=\mbf{1}$, then for all $t\in \R$, the partial derivative $R_y(t\alpha_3, \wt{\chi}\boxtimes \wt{\mu},w_1)$ along the $y$-direction vanishes.
\end{cor}

\subsection{Residues Along $S_1$}
The singularities of $E(\Lambda, \wt{g}, \Phi_\Lambda, \wt{B})$ are the same as those of its constant term

$$E_{\wt{B}}(\Lambda, \wt{g}, \Phi_\Lambda, \wt{B})=\sum_{w\in W} M(\Lambda,\wt{\chi}\boxtimes \wt{\mu},w)\Phi_\Lambda.$$

In view of Proposition \ref{prop: main} on the intertwining map, we see that $\alpha_i$ appears only for $W_i=\set{w\in W: w\alpha_i <0}$. More specifically we have
$$
W_1=\set{w_1, w_{21}, w_{121}, w_{1212}}, \quad W_2=\set{w_2, w_{12}, w_{212}, w_{1212}},$$
$$W_3=\set{w_{12}, w_{121}, w_{212}, w_{1212}}.$$

Therefore, to compute the residue $\Res_{S_1}E_{\wt{B}}(\Lambda, \wt{g}, \Phi_\Lambda, \wt{B})$ we only need to consider terms $M(\Lambda,\wt{\chi}\boxtimes \wt{\mu}, w), w\in W_1$ and take the residues at the hyperplane $S_1$ and combine over all elements of $W_1$. With respect to the $(x, y)$ coordinates $\Lambda=x\alpha_1/2+y\alpha_3/2$, taking the residue along $S_1$ amounts to considering the residue of intertwining operators when $x$ approaches 1. It follows from Proposition \ref{prop: main} that $M(\Lambda,\wt{\chi}\boxtimes \wt{\mu}, w)$, for all $w\in W_1$, has a pole at $S_1$ only when $(\chi\boxtimes \mu)\alpha_1^\vee=\mbf{1}$, i.e. $\chi=\mu$. This gives

\begin{lm}
We use $\Lambda$ to denote $\alpha_1/2+ y\alpha_3/2$. Then the residues along $S_1$ for the intertwining operators in $E_{\wt{B}}(\Lambda, \wt{g}, \Phi_\Lambda, \wt{B})$ are as follows:
\begin{align*}
&\Res_{S_1}M(\Lambda,\wt{\chi}\boxtimes \wt{\mu}, w_1) =\frac{a_{-1}}{\zeta(2)}R(\Lambda,\wt{\chi}\boxtimes \wt{\mu}, w_1)\\
&\Res_{S_1}M(\Lambda,\wt{\chi}\boxtimes \wt{\mu}, w_{21})= \frac{a_{-1}}{\zeta(2)} \frac{L(y+1,\chi^2)}{L(y+2,\chi^2)\varepsilon(y+1,\chi^2)} R(\Lambda,\wt{\chi}\boxtimes \wt{\mu}, w_{21})\\
&\Res_{S_1}M(\Lambda,\wt{\chi}\boxtimes \wt{\mu}, w_{121}) = \frac{a_{-1}}{\zeta(2)} \frac{L(y,\chi^2)}{L(y+2,\chi^2)\varepsilon(y,\chi^2)\varepsilon(y+1,\chi^2)} R(\Lambda,\wt{\chi}\boxtimes \wt{\mu}, w_{121})\\
&\Res_{S_1}M(\Lambda,\wt{\chi}\boxtimes \wt{\mu}, w_{1212})=\frac{a_{-1}}{\zeta(2)} \frac{L(y-1,\chi^2)}{L(y+2,\chi^2)\varepsilon(y-1,\chi^2)\varepsilon(y,\chi^2)\varepsilon(y+1,\chi^2)} R(\Lambda,\wt{\chi}\boxtimes \wt{\mu}, w_{1212}).
\end{align*}
\end{lm}

We require $y\ge0$ from deformation for contour integrals of Eisenstein series. It follows that the pole of $\Res_{S_1}M(\Lambda,\wt{\chi}\boxtimes \wt{\mu}, w), w\in W_1$ appears only for the pairs $(\wt{\chi}\boxtimes \wt{\mu}, y)$ of the three types: 

\begin{itemize}
\item[(a)] $\chi=\mu, \chi^2=\mbf{1}, y=2$;

\item[(b)] $\chi=\mu, \chi^2=\mbf{1}, y=1$;

\item[(c)] $\chi=\mu, \chi^2=\mbf{1}, y=0$.
\end{itemize}
\vspace{0.2cm}

 \cu{First case (a), $\chi=\mu, \chi^2=\mbf{1}, y=2$.} In this case $\Lambda=\alpha_1 /2+\alpha_3$, and only $\Res_{S_1}M(\Lambda,\wt{\chi}\boxtimes\wt{\mu}, w_{1212})$ contributes to the pole. Since the induced representations and local intertwining operators $R(\alpha_1/2+\alpha_3,\wt{\chi}_v\boxtimes \wt{\mu}_v, w_{1212})$ are in the Langlands situation, it follows that the image of this local operator is the Langlands quotient $J_{\wt{B}}(\alpha_1/2+\alpha_3,\wt{\chi}_v\boxtimes \wt{\mu}_v)$  with such prescribed data. In fact, we see it is the even Weil representation for the group $\Mp_4(k_v)$. Now we can write 
$$J_{\wt{B}}(\alpha_1/2+\alpha_3,\wt{\chi}\boxtimes \wt{\mu})=\bigotimes_v 
J_{\wt{B}}(\alpha_1/2+\alpha_3,\wt{\chi}_v\boxtimes \wt{\mu}_v).$$

Since $w_{1212}(\alpha_1/2+\alpha_3)=(-\alpha_1)/2+(-\alpha_3)$, by Langlands criterion the residue $\Res_{\Lambda=\alpha_1/2+\alpha_3}\Res_{S_1}E(\Lambda, \wt{g}, \Phi_\Lambda, \wt{B})$, which we identify with $J_{\wt{B}}(\alpha_1/2+\alpha_3,\wt{\chi}\boxtimes \wt{\mu})$, is square integrable. Let $\msc{B}_1\subseteq \msc{A}(\GL_1(\A_k))\times \msc{A}(\GL_1(\A_k))$ be given by
$$\msc{B}_1=\set{(\chi, \chi): \chi^2=\mbf{1}},$$
then $\bigoplus_{(\chi,\mu)\in \msc{B}_1}J_{\wt{B}}(\alpha_1/2+\alpha_3,\wt{\chi}\boxtimes \wt{\mu})$ contributes to the residual spectrum $L_\text{d}^2(\wt{B})$. \\

\cu{Second case (b), $\chi=\mu, \chi^2=\mbf{1}, y=1$.} In this case $\Lambda=\alpha_4/2$, and simple computation gives
\begin{align*}
\Res_{y=1}\Res_{S_1}E_{\wt{B}}(\Lambda, \wt{g}, \phi, \wt{B}) &=\frac{a_{-1}^2}{\zeta(2)\zeta(3)} R(\alpha_4/2,\wt{\chi}\boxtimes \wt{\mu}, w_{121})\Phi_{\alpha_4/2}-\frac{a_{-1}^2}{\zeta(2)\zeta(3)} R(\alpha_4/2,\wt{\chi}\boxtimes \wt{\mu},w_{1212})\Phi_{\alpha_4/2} \\
&= \frac{a_{-1}^2}{\zeta(2)\zeta(3)} \Bigl(R(\alpha_4/2,\wt{\chi}\boxtimes \wt{\mu},w_{121})\Phi_{\alpha_4/2}-R(\alpha_4/2,\wt{\chi}\boxtimes \wt{\mu},w_{1212})\Phi_{\alpha_4/2} \Bigr).
\end{align*}

By the cocycle relation $R(\Lambda,\wt{\chi}\boxtimes \wt{\mu},w_{1212})=R(w_2\Lambda,w_2(\wt{\chi}\boxtimes \wt{\mu}),w_{121})R(\Lambda,\wt{\chi}\boxtimes \wt{\mu},w_2)$, the last term of above equality equals, up to a constant,
$$R(\alpha_4/2, \wt{\chi}\boxtimes\wt{\mu}, w_{121}) \Bigl(\Phi_{\alpha_4/2}- R(\alpha_4/2, \wt{\chi}\boxtimes \wt{\mu}, w_2)\Phi_{\alpha_4/2} \Bigr).$$

Note that this is valid since $w_2(\alpha_4/2)=\alpha_4/2$ and $w_2(\wt{\chi}\boxtimes\wt{\mu})=\wt{\chi}\boxtimes\wt{\mu}$. It follows from Lemma \ref{lm: id} that $\Res_{z=1/2}\Res_{S_1}E_{\wt{B}}(\Lambda, \wt{g}, \phi, \wt{B})=0$, namely that the point $\alpha_4/2$ does not contribute to the discrete spectrum of $L^2_d(\wt{B})$.\\

\cu{Third case (c), $\chi=\mu, \chi^2=\mbf{1}, y=0$.} In this case $\Lambda=\alpha_1/2$, the residue is given by
$$\Res_{y=0}\Res_{S_1}M(\Lambda,\wt{\chi}\boxtimes \wt{\mu}, w_{21})+\Res_{y=0}\Res_{S_1}M(\Lambda,\wt{\chi}\boxtimes \wt{\mu}, w_{121}),$$
 which vanishes. The argument follows the same as for $y=1$. Such vanishing yielded by Lemma \ref{lm: id} confirms the prediction from general theory of Eisenstein series.

\subsection{Residues Along $S_2$}

We consider the possible poles of intertwining operator for $\Lambda$ along $S_2$, which occur only if $2\angb{\Lambda}{\alpha_2^\vee}=1$ and $(\chi \alpha_2^\vee)^2=\mu^2=\mbf{1}$. Set up another coordinates system $(t, z)$ by
\begin{align}\label{cor_tz}
\Lambda=t\alpha_2/2+z\alpha_4/2 \text{ with }
\left( \begin{array}{c}
y\\
x
\end{array} \right)=
\left( \begin{array}{cc}
1 & 1\\
-1 & 1
\end{array} \right)
\left( \begin{array}{c}
t\\
z
\end{array} \right).
\end{align}

The subset $W_2$ of the Weyl group where intertwining operator can a pole at $S_2$ consists of $w_2, w_{12}$, $w_{212}, w_{1212}$, and we have

\begin{lm}
With respect to the coordinates $(t,z)$, at $t=1/2$ the residues along $S_2$ for each intertwining operator in $E_{\wt{B}}(\Lambda, \wt{g}, \Phi_\Lambda, \wt{B})$ are given by:
\begin{align*}
&\Res_{S_2}M(\Lambda,\wt{\chi}\boxtimes \wt{\mu}, w_2) =\frac{a_{-1}}{2\zeta(2)}R(\Lambda,\wt{\chi}\boxtimes \wt{\mu}, w_2)\\
&\Res_{S_2}M(\Lambda,\wt{\chi}\boxtimes \wt{\mu}, w_{12})= \frac{a_{-1}}{2\zeta(2)} \frac{L(z+1/2,\chi\mu)}{L(z+3/2,\chi\mu)\varepsilon(z+1/2,\chi\mu)} R(\Lambda,\wt{\chi}\boxtimes \wt{\mu}, w_{12})\\
&\Res_{S_2}M(\Lambda,\wt{\chi}\boxtimes \wt{\mu}, w_{212}) = \frac{a_{-1}}{2\zeta(2)} \frac{L(z+1/2,\chi\mu)}{L(z+3/2,\chi\mu)\varepsilon(z+1/2,\chi\mu)}\frac{L(2z,\chi^2)}{L(2z+1,\chi^2)\varepsilon(2z,\chi^2)} R(\Lambda,\wt{\chi}\boxtimes \wt{\mu}, w_{212})\\
&\Res_{S_2}M(\Lambda,\wt{\chi}\boxtimes \wt{\mu}, w_{1212})= \frac{a_{-1}}{2\zeta(2)} \frac{L(z-1/2,\chi\mu)L(2z,\chi^ 2)}{L(z+3/2,\chi\mu)L(2z+1,\chi^2)}\frac{R(\Lambda,\wt{\chi}\boxtimes \wt{\mu}, w_{1212})}{\varepsilon(z-1/2,\chi\mu)\varepsilon(z+1/2,\chi\mu)\varepsilon(2z,\chi^2)} .
\end{align*}
\end{lm}

From the figure, it follows that we need to consider the two cases where $\Res_{S_2}M(\Lambda,\wt{\chi}\boxtimes \wt{\mu}, w), w\in W_2$ may possess a pole:

\begin{itemize}
\item[(i)] $\chi\ne \mu, \chi^2=\mu^2=\mbf{1}, z=1/2$;

\item[(ii)] $\chi=\mu, \chi^2=\mbf{1}, z=1/2$.
\end{itemize}
\vspace{0.2cm}

\cu{First case (i), $\chi\ne \mu, \chi^2=\mu^2=\mbf{1}, z=1/2$.} One simple pole occurs in this case and the residue of constant term along $\wt{B}$ is, up to scalar, given by
\begin{align*}
&\frac{L(1,\chi\mu)}{L(2,\chi\mu)\varepsilon(1,\chi\mu)}R(\alpha_3/2,\wt{\chi}\boxtimes \wt{\mu},w_{212})\Phi_{\alpha_3/2}+
\frac{L(0,\chi\mu)}{L(2,\chi\mu)\varepsilon(0,\chi\mu)\varepsilon(1,\chi\mu)}R(\alpha_3/2,\wt{\chi}\boxtimes \wt{\mu}, w_{1212})\Phi_{\alpha_3/2}\\
&= \frac{L(1,\chi\mu)}{L(2,\chi\mu)\varepsilon(1,\chi\mu)} \Bigl(R(\alpha_3/2,\wt{\chi}\boxtimes \wt{\mu},w_{212})\Phi_{\alpha_3/2}+R(\alpha_3/2,\wt{\chi}\boxtimes \wt{\mu}, w_{1212})\Phi_{\alpha_3/2}\Bigr).
\end{align*}

To determine the image of the intertwining operator $R(\alpha_3/2,\wt{\chi}\boxtimes \wt{\mu},w_{212})+R(\alpha_3/2,\wt{\chi}\boxtimes \wt{\mu}, w_{1212})$, by the cocycle relation
$$R(\Lambda,\wt{\chi}\boxtimes \wt{\mu},w_{1212})=R(w_{212}\Lambda,w_{212}(\wt{\chi}\boxtimes \wt{\mu}),w_1)R(\Lambda,\wt{\chi}\boxtimes \wt{\mu}, w_{212})$$ we see locally

\begin{diagram}
I_{\wt{B}}(\alpha_3/2,\wt{\chi}_v\boxtimes \wt{\mu}_v)  &\rTo^{R(w_{212})}  &I_{\wt{B}}(-\alpha_3/2,\wt{\mu}_v\boxtimes \wt{\chi}_v) &\rTo^{R(w_1)} &I_{\wt{B}}(-\alpha_3/2,\wt{\chi}_v\boxtimes \wt{\mu}_v) \\
 & \rdOnto & \uInto\\
 & &J_{\wt{B}}(\alpha_3/2, \Ind_{\wt{B}}^{\wt{\GL}_2}\wt{\chi}_v\boxtimes \wt{\mu}_v).
\end{diagram}

Here the operator $R(w_1)$ is an isomorphism. Therefore, we see that by taking the sum of two iterated residues one obtains a representation space isomorphic to $\bigotimes_v J_{\wt{B}}(\alpha_3/2, \Ind_{\wt{B}}^{\wt{\GL}_2}\wt{\chi}_v\boxtimes \wt{\mu}_v)$, which we write as $J_{\wt{B}}(\alpha_3/2,\wt{\chi}\boxtimes \wt{\mu})$. Let $\msc{B}_2\subseteq \msc{A}(\GL_1(\A_k))\times \msc{A}(\GL_1(\A_k))$ given by
$$\msc{B}_2=\set{(\chi, \mu): \chi^2=\mu^2=\mbf{1}, \chi\ne \mu}.$$
Since $w_{212}(\alpha_3/2)=w_{1212}(\alpha_3/2)=-\alpha_3/2$, by Langlands criterion, we know $J_{\wt{B}}(\alpha_3/2,\wt{\chi}\boxtimes \wt{\mu})$ is square integrable. Therefore $\bigoplus_{(\chi,\mu)\in \msc{B}_2}J_{\wt{B}}(\alpha_3/2,\wt{\chi}\boxtimes \wt{\mu})$ contributes to $L_\text{d}^2(\wt{B})$.\\

\cu{Second case (ii), $\chi=\mu, \chi^2=\mbf{1}, z=1/2$.} The last three residues in the lemma all have poles, and we will obtain explicit form of $\Res_{\alpha_3/2}\Res_{S_2}E_{\wt{B}}(\Lambda, \wt{g}, \Phi_\Lambda, \wt{B})$. We will see that the residue obtained at $z=1/2$ does not contribute to the discrete spectrum $L^2_\text{d}(\wt{B})$. In fact, it follows from the computation along $S_3$ later that it will be canceled by $\Res_{\alpha_3/2}\Res_{S_3}E_{\wt{B}}(\Lambda, \wt{g}, \Phi_\Lambda, \wt{B})$.

Note that the character $\wt{\chi}\boxtimes \wt{\mu}$ in this case is invariant under $w_1$ and $w_2$, for simplicity we may omit writing $\wt{\chi}\boxtimes \wt{\mu}$ repeatedly for the intertwining operators. First we consider the last two terms in above lemma.

\begin{prop}
Write as before for the number field $k$ its associated zeta function $\zeta(s)=\frac{a_{-1}}{s-1}+a_0 +a_1(s-1)+ ...,$ then with respect to the $(t, z)$ coordinate system in (\ref{cor_tz}) we have

\begin{align*}
&\Res_{z=1/2}\Res_{S_2}M(\Lambda,\wt{\chi}\boxtimes \wt{\mu}, w_{212})+\Res_{z=1/2}\Res_{S_2}M(\Lambda,\wt{\chi}\boxtimes \wt{\mu}, w_{1212})\\
=&\frac{a_{-1}^2}{2\zeta^3(2)}\Big( a_0 R(\alpha_3/2,w_{212})-\frac{a_{-1}}{2}R_x(-\alpha_3/2,w_1)R(\alpha_3/2,w_{212}) \Big).
\end{align*}
Here $R_x(-\alpha_3/2,w_1)$ is the partial derivative along $x$ direction of operator $R(\Lambda, \wt{\chi}\boxtimes \wt{\mu},w_1)$. The coordinates $(x, y)$ for $\Lambda$ are given by (\ref{cor_xy}).
\end{prop}

\begin{proof}
Under the second case (ii), all the characters $\chi\mu, \chi^2$ that appear in the residues are trivial. Therefore, we are dealing with zeta functions. We can get easily
$$\zeta(z+\frac{1}{2})\zeta(2z)=\frac{a_{-1}^2}{2(z-\frac{1}{2})^2} + \frac{3a_0a_{-1}}{2(z-\frac{1}{2})} +...,\text{ and }
\zeta(z-\frac{1}{2})\zeta(2z)=-\frac{a_{-1}^2}{2(z-\frac{1}{2})^2} - \frac{a_0a_{-1}}{2(z-\frac{1}{2})} +....$$
Moreover, from $\Lambda=\alpha_2/4+z\alpha_4/2$ we can write the Taylor expansion
$$R(\Lambda,w_{212})=R(\alpha_3/2,w_{212})+ \Big( R_x(\alpha_3/2,w_{212})+R_y(\alpha_3/2,w_{212})\Big)(z-1/2) + ...$$
and
$$R(w_{212}(\Lambda),w_1)=R(-\alpha_3/2,w_1)+ \Big( R_x(-\alpha_3/2,w_1)- R_y(-\alpha_3/2,w_1)\Big)(z-1/2) + ....$$
Now Lemma \ref{lm: id} gives that $R(-\alpha_3/2,w_1)$ is the identity map and Corollary \ref{cor: partial deriv} that $R_y(-\alpha_3/2,w_1)$ is zero. Hence cocycle relation gives the Taylor expansion
$$R(\Lambda,w_{1212})=R(\alpha_3/2,w_{212})+\Big(R_x(\alpha_3/2,w_{212})+R_y(\alpha_3/2,w_{212})+
R_x(-\alpha_3/2,w_1)R(\alpha_3/2,w_{212})\Big)(z-1/2)+ ....$$

Now the Proposition follows easily from explicit computations of the Taylor expansion of
$M(\Lambda,\wt{\chi}\boxtimes \wt{\mu}, w_{212})$ and $M(\Lambda,\wt{\chi}\boxtimes \wt{\mu}, w_{1212})$.
\end{proof}

\begin{cor}\label{S_2 nonsq}
With respect to the $(x,y)$ coordinate system, we have the iterated residue
$$\Res_{\alpha_3/2}\Res_{S_2}E_{\wt{B}}(\Lambda, \wt{g}, \Phi_\Lambda, \wt{B})=\frac{a_{-1}^2}{\zeta(2)^2}R(\alpha_3/2,w_{12})+\frac{a_{-1}^2}{\zeta(2)^3}\Big( a_0R(\alpha_3/2,w_{212})-\frac{a_{-1}^2}{2} R_x(-\alpha_3/2,w_1)R(\alpha_3/2,w_{212})\Big).$$
This residue is not square integrable and therefore does not contribute to the residual spectrum $L^2_\text{d}(\wt{B})$.
\end{cor}

\begin{proof}
First, note that with respect to the $(t,z)$ coordinates $$\Res_{\alpha_3/2}\Res_{S_2}M(\Lambda, \wt{\chi}\boxtimes \wt{\mu}, w_{12})=\frac{a_{-1}^2}{2\zeta(2)^2}R(\alpha_3/2,w_{12}).$$
Now the transformation of coordinate systems from $(t, z)$ to $(x, y)$ is given by the Jacobian matrix
$$K=\left( \begin{array}{cc}
1 & 1\\
-1 & 1
\end{array} \right).$$
Therefore, the residue computed in the $(x,y)$ system is $\det K$ times the one in the $(t,z)$ system. Thus the result follows from previous Proposition.

For the last assertion, we use the cocycle relation $R(\alpha_3/2,w_{212})=R(w_{12}(\alpha_3/2),w_2)R(\alpha_3/2,w_{12})$. Let $f\in I_{\wt{B}}(\alpha_3/2,\wt{\chi}\boxtimes \wt{\mu})$. If $\Res_{z=1/2}\Res_{S_2}M(\Lambda,\wt{\chi}\boxtimes \wt{\mu}, w_{12})=0$, then
$$\Res_{z=1/2}\Res_{S_2}M(\Lambda,\wt{\chi}\boxtimes \wt{\mu}, w_{212})f+\Res_{z=1/2}\Res_{S_2}M(\Lambda,\wt{\chi}\boxtimes \wt{\mu}, w_{1212})f=0.$$

This refutes the square integrability of the iterated residue of Eisenstein series since $w_{12}(\alpha_3/2)=-\alpha_1/2$. The proof is completed.
\end{proof}

Remark. We will see that such non-square integrable function will be canceled by the residue computed from considering the hyperplane $S_3$.\\

\subsection{Residues Along $S_3$}
Now we compute the iterated residue of Eisenstein series at $\alpha_3$ along the hyperplane $S_3$. We use the original $(x,y)$ coordinate system with $\Lambda=x\alpha_1/2 + y\alpha_3/2$. Then $S_3$ being a singular hyperplane is equivalent to $y=1$ and $(\chi\boxtimes \mu)\alpha_3^\vee=\mbf{1}$, i.e. $\chi=\mu^{-1}$.

Further more, on $S_3$, the point $\alpha_3$ is a singular point for $\Res_{S_2}E_{\wt{B}}(\Lambda, \wt{g}, \Phi_\Lambda, \wt{B})$, to which we deform and get a dimension one continuous spectrum by integration over the imaginary axis. By general residue theory of complex variable (cf. \cite{Ahl} chap. 4 sec. 5), the residue obtained at $\alpha_3/2$ will be \emph{half} the full residue of $\Res_{S_3}E_{\wt{B}}(\Lambda, \wt{g}, \Phi_\Lambda, \wt{B})$ at $\alpha_3/2$. By full residue we mean the coefficient for the leading term in Laurent expansion of $\Res_{S_3}E_{\wt{B}}(\Lambda, \wt{g}, \Phi_\Lambda, \wt{B})$ at $\alpha_3/2$ with respect to $x$.

Recall that $W_3=\set{w: w\alpha_3<0}=\set{w_{12},w_{121},w_{212},w_{1212}}$. Now a simple computation gives

\begin{lm}
With respect to the coordinates $(x,y)$, residues at $y=1$, i.e. along $S_3$, for intertwining operators in $E_{\wt{B}}(\Lambda, \wt{g}, \phi, \wt{B})$ are given by:
\begin{align*}
&\Res_{S_3}M(\Lambda,\wt{\chi}\boxtimes \wt{\mu}, w_{12}) =\frac{a_{-1}}{\zeta(2)} \frac{L(1-x,\mu^2)}{L(2-x,\mu^2)\varepsilon(1-x,\mu^2)} R(\Lambda,\wt{\chi}\boxtimes \wt{\mu}, w_{12})\\
&\Res_{S_3}M(\Lambda,\wt{\chi}\boxtimes \wt{\mu}, w_{121})= \frac{a_{-1}}{\zeta(2)} \frac{L(x,\chi\mu^{-1})}{L(1+x,\chi\mu^{-1})\varepsilon(x,\chi\mu^{-1})}
\frac{L(1+x,\chi^2)}{L(2+x,\chi^2)\varepsilon(1+x,\chi^2)}
R(\Lambda,\wt{\chi}\boxtimes \wt{\mu}, w_{12})\\
&\Res_{S_3}M(\Lambda,\wt{\chi}\boxtimes \wt{\mu}, w_{212}) = \frac{a_{-1}}{\zeta(2)} \frac{L(1-x,\mu^2)}{L(2-x,\mu^2)\varepsilon(1-x,\mu^2)}
\frac{L(1+x,\chi^2)}{L(2+x,\chi^2)\varepsilon(1+x,\chi^2)}
R(\Lambda,\wt{\chi}\boxtimes \wt{\mu}, w_{212})\\
&\Res_{S_3}M(\Lambda,\wt{\chi}\boxtimes \wt{\mu}, w_{1212})= \frac{a_{-1}}{\zeta(2)} \frac{L(x,\chi\mu^{-1})}{L(1+x,\chi\mu^{-1})\varepsilon(x,\chi\mu^{-1})}
\frac{L(1-x,\mu^2)L(1+x,\chi^2)}{L(2-x,\mu^2)\varepsilon(1-x,\mu^2)}
\frac{R(\Lambda,\wt{\chi}\boxtimes \wt{\mu}, w_{212})}{L(2+x,\chi^2)\varepsilon(1+x,\chi^2)}.
\end{align*}

Here $\Lambda=x\alpha_1/2+\alpha_3/2$.
\end{lm}

We see that the meromorphic function $\Res_{S_3}E_{\wt{B}}(\Lambda, \wt{g}, \Phi_\Lambda, \wt{B})$ has possibly a pole at $\alpha_3/2$ or equivalently $x=0$ only when $\chi=\mu, \chi^2=\mbf{1}$. In this case, all $L$-functions are Dedekind zeta functions. Therefore

\begin{prop}
For $\chi=\mu, \chi^2=\mbf{1}$, we have the iterated residue
$$\Res_{\alpha_3/2}\Res_{S_3}E_{\wt{B}}(\Lambda, \wt{g}, \phi, \wt{B})=-\frac{a_{-1}^2}{\zeta(2)^2}R(\alpha_3/2,w_{12})+\frac{a_{-1}^2}{\zeta(2)^3}\Big( \frac{a_{-1}^2}{2} R_x(-\alpha_3/2,w_1)R(\alpha_3/2,w_{212})-
a_0R(\alpha_3/2,w_{212})\Big).$$
In view of Corollary \ref{S_2 nonsq}, we see that under the above condition, there is no residual spectrum generated at $\alpha_3/2$.
\end{prop}

\begin{proof} The proof relies on similar computation as in the $S_2$ case, and we may refrain from writing the character for the intertwining operators since it is invariant $w_1, w_2$.

First, we have 
$$\Res_{x=0}\Res_{S_3}M(\Lambda,\wt{\chi}\boxtimes \wt{\mu}, w_{12})=-\frac{a_{-1}^2}{2\zeta(2)^2}R(\alpha_3/2, w_{12}).$$
Also, by apply the cocycle condition and Lemma \ref{lm: id}, we get
\begin{align*}
\Res_{x=0}\Res_{S_3}M(\Lambda,\wt{\chi}\boxtimes \wt{\mu}, w_{12})& =-\frac{a_{-1}^2}{2\zeta(2)^2}R(\alpha_3/2, w_{121})\\
&=-\frac{a_{-1}^2}{2\zeta(2)^2}R(\alpha_3/2, w_{12}).
\end{align*}

For the last two terms, we first consider the zeta functions that contribute to the poles. Note that
\begin{align*}
\frac{\zeta(x)}{\zeta(1+x)}& =\frac{-\frac{a_{-1}}{x}+a_0-a_1x+...}{-\frac{a_{-1}}{x}+a_0-a_1x+...}\\
& =-1 +\frac{2a_0}{a_{-1}}x+....
\end{align*}

Now the zeta functions that contribute to the pole for the last two intertwining operators are respectively
$$\zeta(1-x)\zeta(1+x)=\frac{-a_{-1}^2}{x^2}+ \frac{0}{x}+..., \text{ and } \frac{\zeta(x)}{\zeta(1+x)}\zeta(1-x)\zeta(1+x)=\frac{a_{-1}^2}{x^2}-\frac{2a_0a_{-1}}{x}+....$$

For the normalized intertwining operator, we can write
\begin{align*}
& R(\Lambda, w_{212})=R(\alpha_3/2, w_{212})+xR_x(\alpha_3/2, w_{212}) +...,\\
& R(w_{212}\Lambda, w_1)=R(-\alpha_3/2, w_1)+xR_x(-\alpha_3/2, w_1) +...,
\end{align*}
which gives
$$R(\Lambda, w_{1212})=R(\alpha_3/2, w_{212})+x\Big( R_x(\alpha_3/2, w_{212})+ R_x(-\alpha_3/2, w_1)R(\alpha_3/2, w_{212}) \Big) +...$$

Now the proposition follows from combining the Taylor expansions, and the consideration that the actual residue $\Res_{x=0}\Res_{S_3}M(\Lambda,\wt{\chi}\boxtimes \wt{\mu}, w_{212})$ and $\Res_{x=0}\Res_{S_3}M(\Lambda,\wt{\chi}\boxtimes \wt{\mu}, w_{12})$ are only half of the coefficient attached to the $1/x$ term in the Taylor expansions of $\Res_{S_3}M(\Lambda,\wt{\chi}\boxtimes \wt{\mu}, w_{212})$ and $\Res_{S_3}M(\Lambda,\wt{\chi}\boxtimes \wt{\mu}, w_{12})$ respectively.
\end{proof}

\subsection{Conclusion}
To summarize,
\begin{thm}\label{thm: borel}
Keep notations as above, it follows from Langlands theory of Eisenstein series that the discrete spectrum $L^2_\text{d}(\wt{B})$ associated to the Borel subgroup has the decomposition
$$L^2_\text{d}(\wt{B})=\bigoplus_{(\chi,\mu)\in \msc{B}_1}J_{\wt{B}}(\alpha_1/2+\alpha_3,\wt{\chi}\boxtimes \wt{\mu}) \oplus \bigoplus_{(\chi,\mu)\in \msc{B}_2}J_{\wt{B}}(\alpha_3/2,\wt{\chi}\boxtimes \wt{\mu}).$$
\end{thm}

\section{Interpretation with Arthur Parameters}
\subsection{The Arthur Conjecture for Metaplectic Groups}

In this section, we depict coarsely how the the residual spectrum computed can be interpreted in the framework of Arthur's conjecture. More precisely, each residue that has appeared in $L^2_\text{res}$ will contribute to one near equivalent class associated to an Arthur parameter $\psi$. We call two automorphic representations are near equivalent if and only if they are isomorphic for almost all places.

Since the emphasis in this paper is not, even if possible, the proof of Arthur conjectures on metaplectic groups, the following sketchy description is not completely unjustified. However, with further work, for example on the construction of local A-packets and proof of full near equivalence class, one can expect to make further progress on the conjecture. \\

We recall that a discrete A-parameter $\psi$ valued in the Langlands dual group $\Sp_{2n}(\C)$ of $\Mp_{2n}$ is
$$\psi: L_k\times \SL_2(\C) \to \Sp_{2n}(\C),$$
which takes the form
$$\psi=\phi_{n_1}\boxtimes S_{r_1} \oplus ... \oplus \phi_{n_k}\boxtimes S_{r_k}.$$
Here $L_k$ is the conjectural automorphic Langlands group. However, we do not need the precise knowledge on it for the Arthur conjecture. For a parameter $\psi$, we require that
\begin{itemize}
\item[(a)] $\phi_{n_i}$ is an irreducible $n_i$-dimensional representation of $L_k$;
\item[(b)] $S_{r_i}$ is the $r_i$-dimensional irreducible representation of $\SL_2(\C)$;
\item[(c)] $\sum_i n_i r_i=2n$;
\item[(d)] $L(s,\phi_{n_i},\bigwedge^2)$ has a pole at $s=1$ if $r_i$ is odd, whereas $L(s,\phi_{n_i},\text{Sym}^2)$ has a pole at $s=1$ if $r_i$ is even;
\item[(e)] $\psi$ is called discrete if it is multiplicity-free, i.e. $(r_i,\phi_{n_i})\ne (r_j,\phi_{n_j})$ if $i\ne j$. It is called tempered if $r_i=1$ for all $i$.
\end{itemize}

We can assign a global component group to every discrete A-parameter $\psi$
$$\mca{S}_\psi=\prod_{i=1}^k \Z/2\Z a_i,$$
which is a vector space over $\Z/2\Z$ with distinguished basis index by the summands of $\psi$. Furthermore, every parameter $\psi$ will give rise to local components by considering the composition
$$\psi_v: L_{k_v}\times \SL(\C) \hookrightarrow L_{k_v}\times \SL(\C) \longrightarrow \Sp_{2n}(\C),$$
where $L_{k_v}=\text{WD}_v$, the Weil-Deligne group for $k_v$. We define the local component group
$$\mca{S}_{\psi_v}=\pi_0(\text{Cent}_{\Sp_{2n}(\C)}(\psi_v)),$$
to which one can attach conjecturally a finite set
$$A_{\psi_v}=\set{\pi_{\xi_v}: \xi_v\in \text{Irr}(\mca{S}_{\psi_v})}$$
indexed by the irreducible characters of the finite component group $A_{\psi_v}$. One of the required conditions for $A_{\psi_v}$ is that: for almost all $v$, $\pi_{\xi_v}$ is irreducible and unramified if $\xi_v=\mbf{1}_v$. In this case, $\pi_{\mbf{1}_v}$ is the unramified representation whose Satake parameter is given by
$$s_{\psi_v}=\psi_v\left( \text{Fr}_v \times  \left(\begin{array}{cc} |\varpi_v|^{1/2} & \\ & |\varpi_v|^{-1/2} \end{array}\right) \right),$$
where $\text{Fr}_v$ and $\varpi_v$ denote the the Frobenius element and local uniformizer for the field $k_v$. In fact, this condition will enable us to decide the parameter $\psi$, associated to which the residual spectrum computed can be interpreted as near equivalent class.

Based on the (conjectural) knowledge of the local packet, we may form the global A-packet
$$A_\psi=\set{\pi_\xi=\otimes \pi_{\xi_v}: \xi=\otimes \xi_v\in \text{Irr}(\mca{S}_{\psi,\A})},$$
which are unitary representations indexed by irreducible characters $\xi$ of the compact group $\mca{S}_{\psi,\A}:=\prod_v \mca{S}_{\psi_v}$. Now the Arthur conjecture can be stated as\\

\textsc{Conjecture}. Denote by $\Delta$ the diagonal map $\mca{S}_\psi \to \mca{S}_{\psi,\A}$. To an A-parameter $\psi$ we can associate a quadratic character $\epsilon_\psi: \mca{S}_\psi \to \set{\pm}$, such that we can define the representation space
$$L_\psi^2=\bigoplus_\xi \angb{\Delta^*(\xi)}{\epsilon}_{\mca{S}_\psi}\pi_\xi,$$
where $\xi$ is taken over all characters of $\mca{S}_{\psi,\A}$. Then the discrete spectrum of $\Mp_{2n}$, when $\psi$ ranges over all A-parameters, has the decomposition
$$L_\text{disc}^2(\Mp_{2n})= \bigoplus_\psi L_\psi^2.$$

\subsection{Residual Spectrum and Near Equivalence Classes}
Now we specialize to the case $n=2$ and interpret the residual spectrum of $\Mp_4$ as near equivalent classes attached to proper A-parameters. Certainly above description of Arthur conjecture provides ample information and redundancy for this purpose, since an analysis of the Satake parameter in each case will be sufficient to give the answer. For our purpose, we have the following types of discrete A-parameters $\psi: L_k\to \Sp_4(\C)$ as from \cite{Art2}.\\

A) \textbf{Tempered case}. Either $\psi=\phi$ (stable case) or $\psi=\phi\oplus \phi^\prime$ (unstable), where for the latter $\phi$ and $\phi^\prime$ are irreducible 2-dimensional representations. This case will be impertinent to our residual spectrum.\\

B) \textbf{Soudry type}, $\psi=\phi \boxtimes S_2$, where $\phi$ is an irreducible representation of such that $L(s, \phi, \text{Sym}^2)$ has a pole at 1. Suppose $\phi$ corresponds to a cuspidal representation $\tau$ of $\GL_2$. For almost all $v$, we know $\pi_{\mbf{1}_v}$ is unramified with Satake parameter

$$\left( \begin{array}{cccc}
\chi_v(\varpi_v)|\varpi_v|^{1/2} & & &\\
 & \mu_v(\varpi_v)|\varpi_v|^{1/2} & &\\
 & & \chi_v(\varpi_v)|\varpi_v|^{-1/2} &\\
&&& \mu_v(\varpi_v)|\varpi_v|^{-1/2}
\end{array} \right),$$
where $\tau_v=\Ind \chi_v\boxtimes \mu_v$ is unramified representation. Here $\varpi_v$ any chosen uniformizer of the local field $k_v$. One see that for almost all $v$, we expect  $\pi_{\mbf{1}_v}\cong J_{\wt{P}_1}(1/2,\wt{\tau}_v)$ and therefore the representation $J_{\wt{P}_1}(1/2,\wt{\tau})$ in Theorem \ref{thm: siegel} belongs to the near equivalence class associated to an A-parameter of Soudry type.\\

C) \textbf{Saito-Kurokawa type}. In this case, we have $\psi=\phi_2\boxtimes S_1\oplus \chi\boxtimes S_2, \chi^2=\mbf{1}$. Let $\pi=\otimes \pi_v$ be the cuspidal representation of $\PGL_2$ that corresponds to $\phi_2$.

For almost all $v$, the local representation $\pi_v=\Ind\mu_v\boxtimes \mu_v^{-1}$ is unramified and therefore $\pi_{\mbf{1}_v}$ has Satake parameter
$$\left( \begin{array}{cccc}
\chi_v(\varpi_v)|\varpi_v|^{1/2} & & &\\
 & \mu_v(\varpi_v) & &\\
 & & \mu_v^{-1}(\varpi_v)&\\
&&& \chi_v(\varpi_v)|\varpi_v|^{-1/2}
\end{array} \right).$$

Therefore, for almost all $v$ one expects $\pi_{\mbf{1}_v}\cong J_\pi(1/2,\wt{\chi}_v\boxtimes \sigma_v)$ and thus $J_\pi(1/2,\wt{\chi}\boxtimes \sigma)$ in Theorem \ref{thm: non-siegel} belongs to the near equivalence class with respect to such $\psi$.\\

D) \textbf{Howe-Piatetski-Shapiro type}, $\psi=\chi\boxtimes S_2 \oplus \mu\boxtimes S_2, \chi\ne \mu, \chi^2=\mu^2=\mbf{1}$. For almost all $v$, the Satake parameter for $\pi_{\mbf{1}_v}$ is given by
$$\left( \begin{array}{cccc}
\chi_v(\varpi_v)|\varpi_v|^{1/2} & & &\\
 & \mu_v(\varpi_v)|\varpi_v|^{1/2} & &\\
 & & \mu_v^{-1}(\varpi_v)|\varpi_v|^{-1/2}&\\
&&& \chi_v(\varpi_v)|\varpi_v|^{-1/2}
\end{array} \right).$$
Therefore for these places $\pi_{\mbf{1}_v}$ is isomorphic to $J_{\wt{B}}(\alpha_3,\wt{\chi}_v\boxtimes \wt{\mu}_v)=J_\text{ETF}(1/2,\wt{\chi}_v\boxtimes \sigma_v)$. The equality of the latter two holds because we can assume $\sigma_v$ is the even Weil representation without loss of generality. In view of this, the residual spectra $J_{\wt{B}}(\alpha_3,\wt{\chi}\boxtimes \wt{\mu})$ in Theorem \ref{thm: borel} and $J_\text{ETF}(1/2,\wt{\chi}\boxtimes \sigma)$ in Theorem \ref{thm: non-siegel} both belong to the near equivalence class for a parameter of Howe-Piatetski-Shapiro type.\\

E) \textbf{Principal type}, $\psi=\chi\boxtimes S_4, \chi^2=\mbf{1}$. For almost all $v$, the Satake parameter for $\pi_{\mbf{1}_v}$ is given by
$$\left( \begin{array}{cccc}
\chi_v(\varpi_v)|\varpi_v|^{3/2} & & &\\
 & \chi_v(\varpi_v)|\varpi_v|^{1/2} & &\\
 & & \chi_v(\varpi_v)|\varpi_v|^{-1/2}&\\
&&& \chi_v(\varpi_v)|\varpi_v|^{-3/2}
\end{array} \right).$$
For such places, we know that $\pi_{\mbf{1}_v}$ is isomorphic to $J_\text{ETF}(3/2,\wt{\chi}_v\boxtimes \sigma_v)=J_{\wt{B}}(\alpha_1/2+\alpha_3,\wt{\chi}_v\boxtimes\wt{\mu}_v)$. The equality is due to same reason as in D). Therefore, the residual spectra $J_\text{ETF}(3/2,\wt{\chi}\boxtimes \sigma)$ in Theorem \ref{thm: non-siegel} and $J_{\wt{B}}(\alpha_1/2+\alpha_3,\wt{\chi}\boxtimes\wt{\mu})$ in Theorem \ref{thm: borel} both belong to the near equivalence class associated to $\psi$ of principal type. In fact, we know they are both nearly equivalent to the $\Mp_4$ analog of elementary theta functions ETF introduced before which occur in the discrete spectrum of $\Mp_2$.

\end{document}